\documentclass{amsart}
\usepackage[all]{xy}
\usepackage{verbatim, rotating, stmaryrd, bbm, pict2e, mathrsfs, amssymb, mathtools, multirow, aliascnt}
\usepackage[colorlinks]{hyperref}

\usepackage{tikz}
\usetikzlibrary{matrix,arrows}
\SelectTips{eu}{}

\makeatletter
\def\eqref{\@ifstar\@eqref\@@eqref}
\def\@eqref#1{\textup{\tagform@{\ref*{#1}}}}
\def\@@eqref#1{\textup{\tagform@{\ref{#1}}}}

\newcommand\@dotsep{4.5}
\def\@tocline#1#2#3#4#5#6#7{\relax
  \ifnum #1>\c@tocdepth
  \else
    \par \addpenalty\@secpenalty\addvspace{#2}%
    \begingroup \hyphenpenalty\@M
    \@ifempty{#4}{%
      \@tempdima\csname r@tocindent\number#1\endcsname\relax
    }{%
      \@tempdima#4\relax
    }%
    \parindent\z@ \leftskip#3\relax \advance\leftskip\@tempdima\relax
    \rightskip\@pnumwidth plus1em \parfillskip-\@pnumwidth
    #5\leavevmode\hskip-\@tempdima #6\relax
    \leaders\hbox{$\m@th
      \mkern \@dotsep mu\hbox{.}\mkern \@dotsep mu$}\hfill
    \hbox to\@pnumwidth{\@tocpagenum{#7}}\par
    \nobreak
    \endgroup
  \fi}

\def\l@subsection{\@tocline{2}{0pt}{20pt}{5pc}{}}
\def\l@subsubsection{\@tocline{2}{0pt}{30pt}{5pc}{}}
\makeatother

\theoremstyle{plain}
\newtheorem{thm}{Theorem}[section]
\newtheorem*{thm*}{Theorem}

\newaliascnt{cor}{thm}
\newaliascnt{prop}{thm}
\newaliascnt{lemma}{thm}
\newaliascnt{sublemma}{thm}
\newtheorem{cor}[cor]{Corollary}
\newtheorem{prop}[prop]{Proposition}
\newtheorem{lemma}[lemma]{Lemma}

\aliascntresetthe{cor}
\aliascntresetthe{prop}
\aliascntresetthe{lemma}
\aliascntresetthe{sublemma}

\theoremstyle{definition}
\newaliascnt{defn}{thm}
\newaliascnt{remark}{thm}
\newaliascnt{example}{thm}
\newaliascnt{notation}{thm}
\newaliascnt{terminology}{thm}
\newaliascnt{observation}{thm}
\newaliascnt{question}{thm}
\newtheorem{defn}[defn]{Definition}
\newtheorem{remark}[remark]{Remark}
\newtheorem*{remark*}{Remark}
\newtheorem{example}[example]{Example}
\newtheorem{notation}[notation]{Notation}

\newtheorem{question*}{Question}
\aliascntresetthe{defn}
\aliascntresetthe{remark}
\aliascntresetthe{example}
\aliascntresetthe{notation}
\aliascntresetthe{terminology}
\aliascntresetthe{observation}
\aliascntresetthe{question}

\numberwithin{equation}{thm}

\newcommand{\bA}{\mathbb A}

\newcommand{\bF}{\mathbb F}
\newcommand{\bG}{\mathbb G}

\newcommand{\bM}{\mathbb M}

\newcommand{\bZ}{\mathbb Z}

\newcommand{\cE}{\mathcal E}

\newcommand{\cN}{\mathcal N}

\newcommand{\cU}{\mathcal U}

\newcommand{\fg}{\mathfrak g}
\newcommand{\fh}{\mathfrak h}

\newcommand{\fu}{\mathfrak u}

\def\Spec{\operatorname{Spec}\nolimits}

\def\Lie{\operatorname{Lie}\nolimits}

\def\dim{\operatorname{dim}\nolimits}

\def\ad{\operatorname{ad}\nolimits}
\def\Grass{\operatorname{Grass}\nolimits}

\def\Stab{\operatorname{Stab}\nolimits}

\def\Spec{\operatorname{Spec}\nolimits}
\def\sl2{\operatorname{SL_{2(2)}}\nolimits}
\def\Ga2{\operatorname{\mathbb G_{\rm a(2)}}\nolimits}
\def\SL{\operatorname{SL}\nolimits}
\def\PGL{\operatorname{PGL}\nolimits}
\def\GL{\operatorname{GL}\nolimits}
\def\Sp{\operatorname{Sp}\nolimits}
\def\SO{\operatorname{SO}\nolimits}

\def\Hom{\operatorname{Hom}\nolimits}

\newcommand{\wt}{\widetilde}

\newcommand{\bE}{\mathbb E}

\DeclareMathOperator{\rank}{rank}
\DeclareMathOperator{\spn}{span}
\DeclareMathOperator{\lt}{LT}
\DeclareMathOperator{\lie}{Lie}

\renewcommand{\max}{\mathrm{Max}}
\newcommand{\set}[1]{\left\{#1\right\}}
\newcommand{\Phir}[1]{\Phi^\mathrm{rad}_{#1}}

\date\today
\begin{document}

\title[Varieties of elementary subalgebras]{Varieties of elementary subalgebras of maximal dimension for modular Lie algebras}

\author[Julia Pevtsova]{Julia Pevtsova$^\dag$}
\address{Department of Mathematics, University of Washington, Seattle, WA}
\email{julia@math.washington.edu}
\thanks{$^\dag$partially supported by the NSF grant 	DMS-0953011}

\author{Jim Stark}
\address{Department of Mathematics, University of Washington, Seattle, WA}
\email{jstarx@math.washington.edu}


\begin{abstract}  Motivated by questions in modular representation theory, Carlson, Friedlander, 
and the first author introduced the varieties $\bE(r, \fg)$ of $r$-dimensional abelian $p$-nilpotent 
subalgebras of a $p$-restricted Lie algebra $\fg$ in \cite{CFP14b}. In this paper, we identify the varieties $\bE(r, \fg)$ 
for a reductive restricted Lie algebra $\fg$ and $r$ the maximal dimension of an abelian $p$-nilpotent 
subalgebra of $\fg$. 
\end{abstract}

\maketitle
\tableofcontents

\setcounter{section}{0}
\section{Introduction}

Let $\fg$ be a restricted Lie algebra defined over a field of positive characteristic $p$. A Lie subalgebra $\cE \subset \fg$ is {\it elementary} if it is abelian with trivial $p$-restriction. The study of the projective variety $\bE(r,\fg)$ of elementary subalgebras of $\fg$ of a fixed dimension $r$ was initiated by Carlson et al.~\cite{CFP14a}. The interest in the geometry of $\bE(r, \fg)$ can be traced back to Quillen's foundational work on mod-$p$ group cohomology which revealed the significance of elementary abelian subgroups of a finite group $G$ to both the representations and cohomology of $G$.  The theory of global nilpotent operators, which associates geometric and sheaf-theoretic invariants living on the space $\bE(r, \fg)$ to representations of $\fg$, was developed in two papers of Carlson et al.~\cite{CFP,CFP14b} and serves as one of the motivations for our interest in $\bE(r, \fg)$.  This variety is also directly related to the much studied variety of $r$-tuples of nilpotent commuting elements in $\fg$ and, consequently, to the cohomology of Frobenius kernels of algebraic groups (see Suslin et al.~\cite{bfsSupportVarieties}). Further geometric properties of $\bE(r, \fg)$ were recently investigated in \cite{warner}. 

In this paper we give a description of the variety $\bE(\fg) = \bE(r_{\rm max}, \fg)$ for $\fg = \Lie G$ the Lie algebra of a reductive algebraic group and $r_{\rm max}$ the maximal dimension of an elementary subalgebra of $\fg$. The maximal dimension of an abelian nilpotent subalgebra of a complex simple Lie algebra $\fg$ is known thanks to the work of Malcev~\cite{malcev} while the general linear case was considered by Schur at the turn of the previous century \cite{schur}.  Malcev has also classified such subalgebras up to automorphisms of the Lie algebra. It turns out that under a mild restriction on $p$ the maximal dimension of an elementary subalgebra in the modular case agrees with Malcev's results.  To compute the variety $\bE(r_{\rm max}, \fg)$ we need to consider elementary subalgebras of $\fg$ up to conjugation by $G$ so our calculations and the end result differ from Malcev's, who classified abelian subalgebras up to automorphisms of $\fg$. Nonetheless, we find his linear algebraic approach very useful for our purposes.  

An analogous classification of elementary abelian $p$-subgroups in a Chevalley group has been considered by several authors, for example Barry~\cite{barryLargeAbelianSubgroups} and Milgram and Priddy~\cite{milgramPriddy}.  Barry classified maximal abelian subalgebras in the $p$-Sylow subgroups of $G(\bF_{p^r})$ for $G$  a  matrix group of classical type ($\SL_n$, $\Sp_{2n}$, $\SO_n$). Using a canonical Springer isomorphism from the nullcone of the Lie algebra $\fg = \Lie G$ to the unipotent variety of $G$, constructed by P. Sobaje \cite{sobajeSpringerIsos}, we recover Barry's results and also obtain analogous classifications for exceptional types. 
We also obtain information on the conjugacy classes of elementary abelian $p$-subgroups of  $G(\bF_{p^r})$. Thanks to the celebrated Quillen stratification theorem this has an immediate application to mod-$p$ group cohomology: The number of conjugacy classes of the elementary abelian $p$-subgroups gives the number of irreducible components of maximal dimension in $\Spec H^*(G(\bF_{p^r}), \overline \bF_p)$. 

The computation of $\bE(\fg) = \bE(r_{\rm max}, \fg)$ reduces to the case of a simple algebraic group $G$ with root system $\Phi$.  For a simple algebraic group, we compute $\bE(\fg)$ under the assumption that $p$ is {\it separably good} for $G$ (see \autoref{defn:sep-good}). We rely on the result of Levy et al.~\cite{lmtNilSubalgebras} to show that any elementary subalgebra of $\fg$ can be conjugated into $\fu \subset \fg$, the Lie algebra of the unipotent radical $U$ of the Borel subgroup $B \leq G$.
The calculation of $\bE(\fg)$ then proceeds in three steps.  First, we explicitly determine $\bE(\fu)$ for $\fu \subset \fg$ {\it as a set}. We define a map $\Lie\colon\max(\Phi) \to \bE(\fu)$ which sends a maximal set of commuting positive roots to an elementary subalgebra of maximal dimension in $\fu$ and show that there is an inverse map $\lt\colon\bE(\fu) \to \max(\Phi)$ which splits $\Lie$.  The map $\Lie$ is not necessarily surjective but we show that for all irreducible root systems $\Phi$ except for $G_2$ and $A_2$ it is surjective up to conjugation by $U$. Hence, we effectively prove that the maximal elementary subalgebras in $\fu$ up to conjugation are given by the combinatorics of the root system of $G$.  This calculation largely relies on the linear algebraic approach of Malcev and is split into several cases dictated by the existence of certain orderings on the corresponding root systems: 
\begin{enumerate}
\item $A_{2n+1}$, $B_2$, $B_3$, $C_n$, $E_7$ (in these cases, there is a unique maximal elementary subalgebra in $\fu$ given by a maximal set of commuting positive roots),
 \item $D_n$ (three maximal elementary subalgebras for $n=4$ and two for $n \geq 5$, all given by maximal sets of commuting positive roots),
 \item $B_n$ (three families of maximal elementary subalgebras for $n=4$, two families for $n \geq 5$, only given by maximal sets of commuting positive roots up to conjugation by $U$),
 \item $G_2$, $A_2$ (exceptional cases).
\end{enumerate}
If we allow conjugation by $G$, then a stronger results holds: with the exception of $G_2$ and $A_2$, every elementary subalgebra in $\bE(\fu)$ is conjugate to a subalgebra stabilized by the action of the Borel.  This observation greatly simplifies the calculation of stabilizers of the conjugacy classes which is the second step in determination of $\bE(\fg)$. The stabilizers of such subalgebras are parabolic subgroups which implies that the $G$-orbits in $\bE(\fg)$ are partial flag varieties. To finish the calculation in all types, except for $G_2$ and $A_2$, we prove in \autoref{thm:sep2} that for $\cE$ with parabolic stabilizer the orbit map $\xymatrix@=10pt{G \ar[r]& G \cdot \cE \subset \bE(\fg)}$  is separable.

When $p$ is not separably good, such as $p=2$ for $G=\PGL_2$, the answer for $\bE(\fg)$ can be somewhat surprising.  In \autoref{ex:pgl2} we utilize a construction from Levy et al.~\cite{lmtNilSubalgebras} to illustrate that.  We are grateful to Jared Warner for pointing out this example to us. We also illustrate in \autoref{ex:G2} that in general the orbit map $G \to G \cdot \epsilon \subset \bE(\fg)$ can fail to be separable. 

The ultimate outcome is that the projective variety $\bE(\fg)$ is a product of $\bE(\fg_i)$ where $\fg_i = \lie(G_i)$ and the $G_i$ range over the simple algebraic subgroups of the derived group of $G$.  When $G$ is simple $\bE(\fg)$ is a finite disjoint union of partial flag varieties {\it unless} $G$ is of type $G_2$ or $A_2$.  This is proved in \hyperref[thm:main]{Theorem~\ref*{thm:main}}:
\begin{thm*}
Let $G$ be a simple algebraic group with root system $\Phi$ which is not of type $A_2$ or $G_2$, and let $\fg = \Lie G$.  Assume that $p$ is separably good for $G$.  
Then 
\[\bE(\fg) = \coprod_{\substack{R \in \max(\Phi) \\ R \ \mathrm{an \ ideal}}}G/P_{R},\]
\end{thm*}
Reinterpreted explicitly for each type, \autoref{thm:main} implies that $\bE(\fg)$ is a disjoint union of at most three copies of generalized Grassmannians in types $A_n$ ($n \not = 2$), $C, D, E$, whereas in types $B$ and $F$ ``two step" partial flag varieties appear (see \autoref{table:Eg}).  For types $A_2$, $G_2$ we show that $\bE(\fg)$ is an irreducible variety and compute its dimension and $G$-orbits. We find that this non-homogeneous answer partially justifies the fact that we have to resort  to case-by-case considerations in our calculations.

We note that calculation of $\bE(\fg)$ for $\fg = \Lie G$ with $G$ a special linear or symplectic group was done in Carlson et al.~\cite{CFP14a}.  The arguments in that paper are based on induction on the dimension and are qualitatively different from the arguments inspired by Malcev's approach used in this paper.

The paper is organized as follows. In \autoref{sec:prelim} we recall various conditions on $p$ such as good, very good, and torsion and define what it means to be separably good.  We also state the combinatorial classification of maximal sets  of commuting roots for irreducible root systems in \autoref{table:max}. We refer to Malcev~\cite{malcev} for this classification but also give a detailed explanation in the Appendix motivated largely by the fact that Malcev's paper only appears to be available in Russian.

The classification of the maximal elementary subalgebras of the unipotent radical $\fu \subset \fg$ up to conjugation by $G$ is settled in \autoref{sec:unip}.  \hyperref[sec:reduc]{Section~\ref*{sec:reduc}}, where we calculate the variety $\bE(\fg)$, contains the main result of the paper.  In \autoref{sec:groups} we apply our results to Chevalley groups, obtaining information on conjugacy classes of maximal elementary abelian $p$-subgroups.

Throughout the paper, $k$  will be an algebraically closed field of positive characteristic
$p$. We follow the convention in Jantzen~\cite{jantzen} and assume that our reductive algebraic $k$-groups are defined and split over $\bZ$.  By Chevalley group we mean a group of the form $G(\bF_{p^r})$ where $G$ is a reductive $k$-group, defined and split over $\bZ$.

{\it Acknowledgments.} The first author is indebted to Eric Friedlander for generously sharing his ideas and intuition about the variety $\bE(r, \fg)$.  The authors would like to thank Steve Mitchell for stimulating discussions at the onset of this project and for pointing out Suter's paper \cite{suterAbIdeals}.  We are also grateful to George McNinch, Paul Sobaje, and Jared Warner for sharing their expertise and patiently answering our structural questions about reductive groups in positive characteristic. 

\section{Preliminaries}
\label{sec:prelim}
\subsection{Notations and conventions} Let $G_{\bZ}$ be a split connected reductive algebraic $\bZ$-group.  Set $G_R = (G_{\bZ})_R$ for any ring $R$ and $G = G_k$.  Let $T_{\bZ} \subseteq G_{\bZ}$ be a split maximal torus and define $T_R$ and $T$ as we did with $G$.  Fix a Borel subgroup $B$ containing $T$ and let $U$ be the unipotent radical of $B$.  Let $X(T) = \Hom(T, \bG_m)$ be the character group and let $\Phi \subseteq X(T)$ be the root system associated to $G$ with respect to $T$.  Let $\Lambda_r(\Phi) = \bZ\Phi$ and $\Lambda(\Phi)$ be the root and integral weight lattices of $\Phi$ respectively.  The quotient $\Lambda(\Phi)/\Lambda_r(\Phi)$ is called the fundamental group of the root system $\Phi$.  If $\Phi$ is irreducible then its fundamental group is cyclic, except for type $D_n$ when $n$ is even, in which case one gets the Klein $4$-group.

If $G$ is semisimple then the quotient $\Lambda(\Phi)/X(T)$ is called the fundamental group of $G$.  We say that $G$ is simply connected if its fundamental group is trivial.   For reductive $G$ we denote by $G_{\mathrm{sc}}$ the simply connected semisimple group of the same type as $G$.  For any subgroup $A$ of the fundamental group of $\Phi$ there is a semisimple group $G$ with root system $\Phi$, fundamental group $A$, and a central isogeny $G_{\mathrm{sc}} \to G$ (see Knus et al.~\cite[25]{involutionsBook} for more details).  The following lemma clarifies the significance of the fundamental group of $G$ for our calculations.

\begin{lemma}[{\cite[2.4]{steinbergTorsion}}]
\label{sep}
Let $G$ be a semisimple algebraic group with root system $\Phi$.  If $p$ does not divide the order of the fundamental group of $G$ then the isogeny $G_{\mathrm{sc}} \to G$ is separable.
\end{lemma}
\begin{proof} Let $H$ be the scheme-theoretic kernel of the isogeny $G_{\mathrm{sc}} \to G$. Then $H$ is a diagonalizable group scheme associated to the finite abelian group $\Lambda(\Phi)/X(T)$ (see, for example, Jantzen~\cite[II.1.6]{jantzen}). In particular, the dimension of the coordinate ring $k[H]$ is equal to the order of $\Lambda(\Phi)/X(T)$. The assumption on $p$ implies that $(p, \dim k[H]) = 1$, hence, $H$ is an \'etale group scheme, $k[H]$ is a separable algebra, and the map $G_{\mathrm{sc}} \to G$ is separable.
\end{proof}

If $\beta = \sum_im_i\alpha_i$ is the highest root written as a linear combination of simple roots then $p$ is \emph{bad} for $\Phi$ if $p = m_i$ for some $i$.  Similarly we may write the dual of this root as a linear combination of dual simple roots $\beta^\vee = \sum_im_i'\alpha_i^\vee$ and $p$ is \emph{torsion} for $\Phi$ if $p = m_i'$ for some $i$.  A prime is \emph{good} (respectively \emph{non-torsion}) if it is not bad (respectively torsion).  We say $p$ is \emph{very good} for $\Phi$ if $p$ is good for $\Phi$ and $p$ does not divide the order of the fundamental group of $\Phi$.  Some authors include this condition in the definition of non-torsion but we will not.  We instead say that a prime is \emph{very non-torsion} if it is non-torsion and $p$ does not divide the order of the fundamental group of $\Phi$.

\begin{defn}
\label{defn:sep-good}
If $G$ is a semisimple algebraic group we say that $p$ is {\it separably good} for $G$ if
\begin{enumerate}
\item $p$ is good for $G$,
\item the isogeny $G_{\mathrm{sc}} \to G$ is separable. 
\end{enumerate} 
If $G$ is a connected reductive group we say that $p$ is \emph{separably good} for $G$ if it is separably good for its derived group $[G, G]$.
\end{defn}

Note that very good implies separably good by \autoref{sep} but in type $A$ the separably good condition is less restrictive.  The simple groups of type $A_{n - 1}$ are $\SL_n/\mu_d$ for $d \mid n$ and by \autoref{sep} if $p \nmid d$ then $p$ is separably good.  In particular, $p$ is always separably good if $d = 1$ so the Special Linear group $\SL_n$ is covered by our results for all $p$.

Let $\fg$ and $\fu$ be the Lie algebras of $G$ and $U$ respectively.  For any ring $R$ let $\fg_R = \Lie G_R$ so that $\fg = \fg_k$.  Note that if $R \subseteq S$ are rings then one has $\fg_S = \fg_R \otimes_R S$.

\begin{defn}
Let $R \subseteq k$ be a subring.  An element $x \in \fg$ is \emph{defined over $R$} if there exists an $x' \in \fg_R$ such that $x = x' \otimes 1$.  A Lie subalgebra $\fh \subseteq \fg$ is \emph{defined over $R$} if there exists a Lie subalgebra $\fh' \subseteq \fg_R$ such that $\fh = \fh' \otimes_R k$.  We will call $x'$ and $\fh'$ \emph{$R$-forms} of $x$ and $\fh$ respectively.
\end{defn}

For a simply connected semisimple group $G$ the Lie algebra $\fg$ has a Chevalley basis $\set{x_\alpha, h_i \ | \ \alpha \in \Phi, \ 1 \leq i \leq \rank\Phi}$ defined over $\bZ$.  In particular, this means $[x_\alpha, x_\beta] = N_{\alpha, \beta}x_{\alpha + \beta}$, where $N_{\alpha_\beta} = 0$ if $\alpha + \beta \notin \Phi$ and when $\alpha + \beta \in \Phi$ we have $N_{\alpha, \beta} = \pm(r + 1)$ where $-r\alpha + \beta, \ldots, s\alpha + \beta$ is the $\alpha$-string through $\beta$.  The sign can be inductively determined depending on a choice of ordering for the roots and our choice that $N_{\alpha, \beta} = +(r + 1)$ when $(\alpha, \beta)$ is an extraspecial pair defined by this ordering (see Carter~\cite[4.2]{carterSGLT} for details).

For a general reductive group there exists a central isogeny $G_{\mathrm{sc}} \times D \to G$ where $D$ is some torus.  The preimage of the Borel $B \subseteq G$ is a Borel in the reductive group $G_{\mathrm{sc}} \times D$ and the isogeny restricts to an isomorphism between the unipotent parts of these Borels.  We may therefore define $x_\alpha \in \fg$ and $h_i = [x_{\alpha_i}, x_{-\alpha_i}]$ as the image under the isogeny of the corresponding elements of $\lie G_{\mathrm{sc}}$.  It may be the case that the $h_i$ do not span $\lie(T)$, nevertheless, $\set{x_\alpha \ | \ \alpha \in \Phi^+}$ is still a basis for $\fu$ and satisfies the same relations as in the simply connected case.  For our purposes the torus is somewhat irrelevant, so when this is the case we append to the set $\set{x_\alpha \ | \ \alpha \in \Phi}$ any basis of $\lie(T)$ and, with some abuse of terminology, call the result a Chevalley basis of $\fg$.

\begin{defn}
A subalgebra $\cE \subseteq \fg$ is called \emph{Chevalley} if it is spanned by some subset of a Chevalley basis.
\end{defn}

Finally, we note that good primes are greater than or equal to the length of the longest root string in the root system $\Phi$.  In particular, this implies that the structure constants $N_{\alpha, \beta} \in \bZ$ for the Chevalley basis are not divisible by $p$.  Thus in $\fg$ one has $[x_\alpha, x_\beta] \neq 0$ if and only if $\alpha + \beta \in \Phi$.

\begin{table}[ht]
\caption{Bad and torsion primes, fundamental groups, and maximal root string lengths.~\cite[2.13]{geckMalle}~\cite[9.2]{malletesterman}}
\centering
\begin{tabular}[b]{|r||c|c|c|c|c|c|c|c|c|}
\hline
Type & $A_n$ & \parbox{33pt}{\vspace{2pt}\centering$B_n$ \\ $(n \geq 2)$\vspace{2pt}} & \parbox{33pt}{\vspace{2pt}\centering$C_n$ \\ $(n \geq 3)$\vspace{2pt}} & \parbox{33pt}{\vspace{2pt}\centering$D_n$ \\ $(n \geq 4)$\vspace{2pt}} & $E_6$ & $E_7$ & $E_8$ & $F_4$ & $G_2$ \\
\hline
Bad & none & $2$ & $2$ & $2$ & $2, 3$ & $2, 3$ & $2, 3, 5$ & $2, 3$ & $2, 3$ \\
\hline
Torsion & none & $2$ & none & $2$ & $2, 3$ & $2, 3$ & $2, 3, 5$ & $2, 3$ & $2$ \\
\hline
$|\Lambda/\Lambda_r|$ & $n + 1$ & $2$ & $2$ & $4$ & $3$ & $2$ & $1$ & $1$ & $1$ \\
\hline
\parbox{55pt}{\vspace{2pt}Longest root string length\vspace{1pt}} & $2$ & $3$ & $3$ & $2$ & $2$ & $2$ & $2$ & $3$ & $4$ \\
\hline
\end{tabular}
\label{table:primes}
\end{table}

\subsection{Maximal sets of commuting roots}

Let $\Phi$ be an irreducible root system with positive roots $\Phi^+$, corresponding base $\Delta = \{\alpha_1, \ldots, \alpha_n\}$, and the Weyl group $W$. Throughout the paper, we follow the labeling in Bourbaki~\cite[6, \S 4]{bourbakiLie2}.  Let $S \subseteq \Delta$ be a set of simple roots and $\overline S = \Delta \setminus S$ its complement.  Then we define
\[
\Phir S = \Phi^+ \setminus \mathbb N\overline S
\]
to be the positive roots that cannot be written as a linear combination of the simple roots not in $S$.  Note that $S \mapsto \Phir S$ commutes with unions and intersections.  If $S = \set{\alpha_i}$ then we will write $\Phir i$ instead of $\Phir{\set{\alpha_i}}$.

For any $I \subset \Delta$ define the parabolic subgroup $W_I$ and its corresponding root system $\Phi_I$ as in Humphreys~\cite[1.10]{humphreysReflGrps}.  For $I = \Delta \setminus S$, we have $\Phir{S} = \Phi^+ \setminus \Phi_I^+$.  If $P_I = L_I \ltimes U_I$ is the standard parabolic determined by the subset $I = \Delta \setminus S$ with the Levi factor $L_I$ and the unipotent radical $U_I$, then the root subgroups $U_\alpha$ with $\alpha \in \Phir{S}$ are precisely the ones generating $U_I$.

\begin{lemma} \label{prop:radStab}
Let $S \subseteq \Delta$.  Then $\Stab_W(\Phir{S}) = W_{\Delta \setminus S}$.
\end{lemma}
\begin{proof}
Let $I = \Delta \setminus S$.   As $W_I$ stabilizes $\Phi$ and $\Phi_I$, it also stabilizes $\Phi \setminus \Phi_I = \Phir{S} \cup -\Phir{S}$.  Now for any element $w \in W_I$ its length as an element of $W$ equals its length as an element of $W_I$.  Length is characterized by the number of positive roots that are sent to negative roots so such roots are in $\Phi_I$.  In particular, $w\Phir{S} \subseteq \Phir{S} \cup -\Phir{S}$ are positive so $w\Phir{S} \subseteq \Phir{S}$.  This proves that $W_I$ stabilizes $\Phir{S}$.

Conversely assume $w \in W \setminus W_I$ is of minimal length stabilizing $\Phir{S}$.  Then $w$ stabilizes $\Phi_I$ and for any $\alpha \in I$ we have $\ell(ws_\alpha) = \ell(w) + 1$ so $w(\alpha)$ is positive.  This means $w$ stabilizes $\Phi_I^+$, so it stabilizes all of $\Phi^+$, but the identity is the only such element in $W$.
\end{proof}

\begin{lemma} \label{lem:radConj}
The set $\Phir{S}$ is not conjugate to any other set of positive roots.
\end{lemma}
\begin{proof}
If $w$ does not stabilize $\Phir{S}$ then it may be written in the form $w = us_{\alpha_i}v$ where $v \in W_{\Delta \setminus S}$ and $s_{\alpha_i} \notin W_{\Delta \setminus S}$, i.e., $\alpha_i \in S$.  By Humphreys~\cite[1.7]{humphreysReflGrps} we have $us_{\alpha_i}(\alpha_i) < 0$ so $w$ sends $v^{-1}(\alpha_i) \in \Phir{S}$ to a negative root.
\end{proof}

Two positive roots \emph{commute} if their sum is not a root.  A set of commuting roots is a set of positive roots which pairwise commute and $R \subseteq \Phi^+$ is a \emph{maximal} set of commuting roots if it is maximal with respect to order, i.e., if $R'$ is any other set of commuting roots then $|R'| \leq |R|$.

\begin{notation} Let $\max(\Phi)$ be the set of all maximal sets of commuting roots in $\Phi$.  Let $m(\Phi)$ be the order of a maximal set of commuting roots in $\Phi^+$.  If $\Phi$ is irreducible of type $T$ then we may write $\max(T)$ and $m(T)$ instead.
\end{notation}

To formulate the theorem on the maximal sets of commuting roots, we need to introduce additional notation for type $B_n$. We first recall notation from Bourbaki~\cite[6, \S 4.5]{bourbakiLie2}:
\begin{notation}\label{not:Bn} Type $B_n$.
\begin{align*}
\epsilon_i & = \ \alpha_i + \alpha_{i+1} + \cdots + \alpha_n, & \quad 1 \leq i \leq n\\
\epsilon_i + \epsilon_j & = \  (\alpha_i + \alpha_{i+1} + \cdots + \alpha_n) + (\alpha_j + \alpha_{j+1} + \cdots + \alpha_n), & \quad 1 \leq i < j \leq n\\
\epsilon_i - \epsilon_j & = \ \alpha_i + \alpha_{i+1} + \cdots + \alpha_{j-1}, & \quad 1 \leq i < j \leq n
\end{align*} 

\noindent
Define the following subsets of positive roots, as in Malcev~\cite{malcev}.
\begin{align*}
S_t &= \{\epsilon_t, \epsilon_i + \epsilon_j \, | \, 1 \leq i < j \leq n\}, & \quad t = 1, 2, \ldots, n \\
S^*_t &= \{\epsilon_t, \epsilon_i + \epsilon_j, \epsilon_{i^\prime} - \epsilon_n \, | \, 1 \leq i < j < n, 1 \leq i^\prime <n\}, & \quad t = 1, 2, \ldots, n  
\end{align*}
\end{notation}
We present what is known about the maximal sets of positive roots in \autoref{table:max}. The proofs can be found in Malcev~\cite{malcev}, see also \autoref{Appendix}.

\renewcommand{\arraystretch}{1.5}

\begin{table}[ht]
\caption{Maximal sets of commuting roots.}
\centering
\begin{tabular}[b]{|c|c|c|c|c|}
\hline
{\rm Type} $T$ & \parbox{50pt}{\vspace{3pt}\centering Restrictions on rank\vspace{3pt}}  & $\max(T)$ & $\#\max(T)$ & $m(T)$ \\
\hline \hline
$A_{2n}$ & $n \geq 1$ & $\Phir{n+1}$, $\Phir{n}$ & $2$ & $n(n+1)$ \\
\hline
$A_{2n+1}$ & $n \geq 0$ & $\Phir{n+1}$ & $1$ & $(n+1)^2$ \\
\hline
\multirow{6}{*}{$B_n$}& $n=2,3$ & $\Phir{1}$ & $1$ & $2n-1$ \\
\cline{2-5}
& $n=4$ & $\begin{array}{c} \Phir{1},\\S_1, S_2, S_3, S_4, \\ S^*_1, S^*_2, S^*_3 \end{array}$ & $8$ & $7$ \\
\cline{2-5}
& $n \geq 5$ & $\begin{array}{c} S_t, 1 \leq t \leq n \\ S_t^*, 1 \leq t  < n\end{array}$& $2n-1$ & $\frac{1}{2}n(n-1)+1$\\
\hline
$C_n$ & $n \geq 3$ & $\Phir{n}$ & $1$ & $\frac{1}{2}n(n+1)$  \\
\hline
\multirow{2}{*}{$D_n$}& $n=4$& $\Phir{1}, \Phir{3}, \Phir{4}$ & $3$ & $6$ \\
\cline{2-5}
& $n \geq 5$ & $ \Phir{n-1}, \Phir{n}$ & $2$ & $\frac{1}{2}n(n-1)$ \\
\hline
$E_6$ &&$\Phir{1}, \Phir{6}$&$2$&$16$ \\
\hline
$E_7$ &&$\Phir{7}$&1&27 \\
\hline
$E_8$ &&\begin{tabular}{p{0.7in}} none of the form $\Phir{i}$ \end{tabular}&$134$&$36$ \\
\hline
$F_4$ &&\begin{tabular}{p{0.7in}} none of the form $\Phir{i}$ \end{tabular}&$28$&$9$ \\
\hline
$G_2$ &&\begin{tabular}{p{0.7in}} none of the form $\Phir{i}$ \end{tabular}&$5$&$3$ \\
\hline
\end{tabular}
\label{table:max}
\end{table}

\begin{remark}[On cominuscule roots] Note that in types $A$, $C$, $D$, $E_6$, $E_7$, the sets of maximal commuting roots are given by $\Phir{i}$ with $\alpha_i$ a simple cominuscule root (see Billey and Lakshmibai~\cite{Billey-Lakshmibai} or Richardson et al.~\cite{Richardson-Rohle-Steinberg} for more on cominuscule roots).  One of the equivalent definitions of a simple cominuscule root $\alpha_i$ is that the unipotent radical of the corresponding parabolic is abelian \cite[Lem 2.2]{Richardson-Rohle-Steinberg} so the set $\Phir{i}$ is a natural candidate to be in $\max(T)$. As one sees from the table, the sets $\Phir{i}$ defined by a simple cominuscule root do have maximal dimension in almost all cases when they exist {\it except} for the most mysterious case of $B_n$.
\end{remark}

\begin{defn}
We say that $R \subseteq \Phi^+$ is an \emph{ideal} if $\alpha + \beta \in R$ whenever $\alpha \in \Phi^+$, $\beta \in R$, and $\alpha + \beta \in \Phi^+$.
\end{defn}

Note, for example, that the sets $\Phir{i}$ are always ideals.  The computations below require knowing which maximal sets of commuting roots are ideals and for such sets $R$ what is the stabilizer $\Stab_W(R) = \set{w \in W \ | \ wR = R}$.  This information can be found in \autoref{table:stab} (see \autoref{Appendix} for the calculation).

Note that $\max(\Phi)$ contains non-ideals only in types $B_n (n \geq 4)$, $E_8$, $F_4$, and $G_2$.  One can check in types $B_n$, $E_8$, and $F_4$ that every set in $\max(\Phi)$ is $W$-conjugatate to an ideal in $\max(\Phi)$.  In any type there is at most one ideal that is not of the form $\Phir{i}$ for some $i$, therefore in all types except $G_2$ \autoref{lem:radConj} gives that each set in $\max(\Phi)$ is conjugate to a \emph{unique} ideal in $\max(\Phi)$.  The exceptional case is $G_2$, where one finds that there are two orbits under the partial action of $W$ on $\max(\Phi)$ and only one contains an ideal.

\begin{table}[ht]
\caption{Maximal commuting ideals and their stabilizers}
\centering
\begin{tabular}[b]{|c|c|c|c|}
\hline
Type $T$ & Ideal $R$ & $\Stab_W(R)$ \\
\hline \hline
Any & $\Phir{i}$ & $W_{\Delta \setminus \set{\alpha_i}}$ \\
\hline
$B_n$, $n \geq 4$ & $S_1$ & $W_{\Delta \setminus \set{\alpha_1, \alpha_n}}$ \\
\hline
$E_8$ & \begin{tabular}{p{0.7in}} $R$ is unique \end{tabular} & $W_{\Delta \setminus \set{\alpha_2}}$ \\
\hline
$F_4$ & \begin{tabular}{p{0.7in}} $R$ is unique \end{tabular} & $W_{\set{\alpha_1, \alpha_3}}$ \\
\hline
$G_2$ & \begin{tabular}{p{0.7in}} $R$ is unique \end{tabular} & $W_{\set{\alpha_2}}$ \\
\hline
\end{tabular}
\label{table:stab}
\end{table}

\subsection{Varieties $\bE(r, \fg)$}

\begin{defn}[Carlson et al.~\cite{CFP14a}]
An elementary subalgebra $\cE \subset \fg$ of dimension $r$
is a $p$-restricted Lie subalgebra of dimension $r$ which is commutative and has $p$-restriction equal to $0$.
We define
\[\bE(r,\fg) = \set{\cE \subset \fg \ | \ \cE \ \text{elementary subalgebra of dimension} \ r}\]
\end{defn}
So defined, $\bE(r, \fg)$ is a closed subset of the Grassmannian of $r$-planes in the vector space $\mathfrak g$ and hence has a natural structure of a projective algebraic variety.  In this paper, we are concerned with the varieties of elementary subalgebras of maximal dimension.  We set
\begin{align*}
r_{\rm max} &= \max\set{r \ | \ \bE(r, \fg) \ \text{is nonempty}}, \\
\bE(\fg) &= \bE(r_{\rm max}, \fg).
\end{align*}
\begin{thm} \label{thm:Eprod}
Let $\fg$ and $\fh$ be restricted Lie algebras whose maximal elementary abelians have dimensions $r$ and $s$ respectively.  Then $r + s$ is the dimension of a maximal elementary abelian in $\fg \oplus \fh$ and $\bE(\fg \oplus \fh) \simeq \bE(\fg) \times \bE(\fh)$.
\end{thm}
\begin{proof}
If $\cE \subseteq \fg \oplus \fh$ is maximal and $\cE_1 \subseteq \fg$ and $\cE_2 \subseteq \fh$ are its images under the projections to $\fg$ and $\fh$ respectively then $\cE_1 \oplus \cE_2 \subseteq \fg \oplus \fh$ is elementary abelian and contains $\cE$, hence equals $\cE$.  This proves that every maximal elementary abelian is a sum of necessarily maximal elementary abelians from $\fg$ and $\fh$.  Thus $r + s$ is the maximal dimension of an elementary abelian in $\fg \oplus \fh$ and there is a bijection $\bE(\fg) \times \bE(\fh) \to \bE(\fg \oplus \fh)$.  One then checks on the standard affine open sets of the Grassmannian that this is an isomorphism.
\end{proof}


\section{Unipotent case} \label{sec:unip}

In this section we assume that $G$ is a simple algebraic group and we compute $\mathbb E(\mathfrak u)$ as a {\it set}.  We do this by defining the leading terms associated to a particular subalgebra as a subset of the root system and showing that such a subset must be a maximal set of commuting roots.  We then perform a case by case analysis based on the information in \autoref{table:max} which shows that in most cases $\bE(\fu) = \lie(\max(\Phi))$, with $\lie(\max(\Phi))$ defined in \autoref{eq:lie}.  The exceptions are type $A_2$ where $\bE(\fu) \simeq \mathbb P^1$, type $B_n$ for $n \geq 4$ where $\bE(\fu)$ includes the lie subalgebras $B(a_1, \ldots, a_n)$ and $\exp(\ad(a_nx_{\alpha_n}))(C(a_1, \ldots, a_{n - 1}))$ defined below, and the exceptional types of which only $G_2$ is calculated explicitly.  For types $F$ and $E$ we verify later that under the adjoint action one has $\bE(\fu) \subseteq G\cdot\lie(\max(\Phi))$.  We note that the condition $\bE(\fu) \subseteq G\cdot\lie(\max(\Phi))$ holds in all types except $G_2$.

Though we are interested in the case when $p$ is separably good, the results in this section are valid so long as $p$ is greater than or equal to the length of the longest root string in $\Phi$.

\subsection{Correspondence with sets of commuting roots}
We must first choose a total ordering $\preceq$ on $\Phi^+$ which respects addition of positive roots, that is, if $\beta, \gamma, \lambda, \beta + \lambda, \gamma + \lambda \in \Phi^+$ and $\beta \preceq \gamma$ then $\beta + \lambda \preceq \gamma + \lambda$.  In examples the choice of ordering will depend on the root system but we note here that such a total ordering always exists: The standard ordering $\leq$ on $\Phi$ respects addition, as does a reverse lexicographical ordering with respect to any ordering of the simple roots.  This ordering will define the extraspecial pairs in our root system and consequently the signs in the structure constants of our Chevalley basis.  We will often construct such orderings through refinement.

\begin{defn}
Let $\succeq_1, \succeq_2, \ldots, \succeq_n$ be relations on a set $X$ which define $\succ_i$ and $=_i$ in the obvious way.  The relation
\[\succeq \ = (\succeq_1, \succeq_2, \ldots, \succeq_n)\]
is defined as follows: $x \succeq y$ if either
\begin{enumerate}
\item For some $1 \leq i \leq n$ we have $x \succ_i y$ and $x =_j y$ for all $j < i$, or
\item $x =_i y$ for all $1 \leq i \leq n$.
\end{enumerate}
We say that $\succeq$ is given by \emph{refining} $\succeq_1$, first by $\succeq_2$, then by $\succeq_3$, and so on.
\end{defn}
In simple terms we compare by first trying $\succeq_1$ and inductively trying $\succeq_{i + 1}$ if $\succeq_i$ gives equality.  We will always choose the $\succeq_i$ to be preorders which respect addition of positive roots and we will choose $\succeq_n$ to be, moreover, a total order.  Then $\succeq$ is a total order which respects addition.

Now for each set $R \subseteq \Phi^+$ of commuting roots there is an abelian Chevalley subalgebra $\lie(R) = \spn_k\set{x_\beta \ | \ \beta \in R}$ associated to $R$ with $\dim\lie(R) = |R|$.  As elements $x_\beta$ in the Chevalley basis are always $p$-nilpotent this is in fact an elementary abelian subalgebra.  We wish to show that this induces a map
\begin{equation}
\label{eq:lie}
\xymatrix@=15pt{\lie\colon\max(\Phi) \ar[r]& \mathbb E(\mathfrak u),}
\end{equation}
that is, that the Lie subalgebra associated to a maximal commuting set of roots has maximal dimension among all elemenatry subalgebras.  We do this using an argument of Malcev~\cite{malcev} which shows that there exists a surjection
\begin{equation}
\label{eq:lt}
\xymatrix@=15pt{\lt: \bE(\fu) \ar[r]& \max(\Phi)}
\end{equation}
that splits $\lie$.

Let $\cE \subseteq \mathfrak u$ be an elementary subalgebra.  The ordering $\succeq$ on $\Phi^+$ gives an ordering on the basis elements $x_\beta$ of $\mathfrak u$.  Choose the unique basis of $\cE$ which is in reduced echelon form with respect to this ordering and let $\lt(\cE)$ be the set of roots $\beta$ such that the corresponding $x_\beta$ are the leading terms in this reduced basis.  Observe that if $\beta$ and $\gamma$ are the leading terms of $b_1 = x_\beta + \langle\text{lower terms}\rangle$ and $b_2 = x_\gamma + \langle\text{lower terms}\rangle$ respectively, and if $\beta + \gamma \in \Phi^+$ then $[x_\beta, x_\gamma] = N_{\beta, \gamma}x_{\beta + \gamma}$ is the leading term of $[b_1, b_2]$.  Thus if $[b_1, b_2] = 0$ then $\beta$ and $\gamma$ commute.  This proves that $\lt(\cE)$ is a commuting set of roots.  Clearly $\lt(\lie(R)) = R$ so $\lt$ splits $\lie$ and both maps preserve maximality.

\subsection{Types $A_{2n + 1}$, $B_2$, $B_3$, $C_n$, and $E_7$}

In these types there is a unique maximal set of commuting roots of the form $\Phir{i}$ for some $i$.  Let $\succeq$ be the reverse lexicographic ordering given by $\alpha_i < \alpha_1 < \alpha_2 < \cdots$, so $\beta \succeq \gamma$ if when written as a linear combination of simple roots the coefficient of $\alpha_i$ in $\gamma$ is larger than the coefficient of $\alpha_i$ in $\beta$, or if those coefficients are equal then coefficient of $\alpha_1$ in $\gamma$ is larger, and so on.  With this ordering the roots in $\Phir{i}$, which all have nonzero $\alpha_i$ coefficient, are strictly smaller than the roots in $\Phi^+ \setminus \Phir{i}$, which have $0$ as the $\alpha_i$ coefficient.  The following lemma then gives that $\lie(\Phir{i})$ is the only possible maximal elementary subalgebra.

\begin{lemma} \label{lemOto}
If $\Phi^+ \setminus \lt(\cE) \succ \lt(\cE)$ then $\cE = \lie(\lt(\cE))$.
\end{lemma}
\begin{proof}
Choose $\beta \in \lt(\cE)$ and let $b$ the element in the reduced echelon form basis of $\cE$ with leading term $x_\beta$.  The terms in $b - x_\beta$ are of the form $cx_\gamma$ with $\beta \succeq \gamma$, equivalently, $\gamma \in \lt(\cE)$.  So $c = 0$ because $x_\gamma$ is the leading term of some other basis element.  Thus we have $b = x_\beta$ and our reduced basis is $\set{x_\beta \ | \ \beta \in \lt(\cE)}$.
\end{proof}

\subsection{Type $A_2$}

Along with type $G_2$ this case will be exceptional in that it is not true that every elementary subalgebra is conjugate to a subalgebra in $\lie(\max(\Phi))$ and it is not true that every Weyl group orbit in $\max(\Phi)$ contains an ideal.  Consequently we will find below that the variety $\mathbb E(\mathfrak g)$ for a simple algebraic group of type $A_2$ is not a disjoint union of flag varieties.

The highest root $\alpha_1 + \alpha_2$ commutes with all positive roots so $x_{\alpha_1 + \alpha_2}$ is contained in any maximal elementary $\cE$.  As the dimension of such $\cE$ is $2$ this means $\cE$ may be generated by $x_{\alpha_1 + \alpha_2}$ and an element of the form $ax_{\alpha_1} + bx_{\alpha_2}$ where $a, b \in k$ are not both $0$.  One can check that no other conditions on $a$ and $b$ are needed to get an elementary abelian therefore we have a bijection
\begin{align*}
\mathbb P^1 &\simeq \bE(\fu) \\
[a : b] &\mapsto \langle ax_{\alpha_1} + bx_{\alpha_2}, x_{\alpha_1 + \alpha_2}\rangle.
\end{align*}
We note that each of these subalgebras is fixed under conjugation by $U < G$, conjugation by the torus allows us to scale the constants $a$ and $b$, and no two of these subalgebras are conjugate via a representative of a Weyl group element.  Thus up to conjugation there are three subalgebras:
\begin{align*}
& \langle x_{\alpha_1}, x_{\alpha_1 + \alpha_2}\rangle, \\
& \langle x_{\alpha_2}, x_{\alpha_1 + \alpha_2}\rangle, \\
& \langle x_{\alpha_1} + x_{\alpha_2}, x_{\alpha_1 + \alpha_2}\rangle.
\end{align*}

\subsection{Type $A_{2n}$, $n \geq 2$}

Consider positive roots $\beta$ and $\gamma$ written as a linear combination of the simple roots $\alpha_i$.  Define $\beta \succeq_1 \gamma$ if the sum of the coefficients of $\alpha_n$ and $\alpha_{n + 1}$ in the expression for $\gamma$ is greater or equal the sum for $\beta$.  Define $\succeq_2$ to be the reverse lexicographic ordering given by $\alpha_{n + 1} \prec \alpha_n \prec \alpha_1 \prec \alpha_2 \prec \cdots$.  Finally let $\succeq$ be the refinement of $\succeq_1$ by $\succeq_2$.  One can now check that
\[\Phi^+ \setminus (\Phir{n} \cup \Phir{n + 1}) \ \succ \ \Phir{n} \setminus \Phir{n + 1} \ \succ \ \Phir{n + 1} \setminus \Phir{n} \ \succ \ \Phir{n} \cap \Phir{n + 1}.\]

Let $\cE \subseteq \mathfrak u$ be an elementary subalgebra.  If $\lt(\cE) = \Phir{n + 1}$ then \autoref{lemOto} gives $\cE = \lie(\Phir{n + 1})$.  Assume $\lt(\cE) = \Phir{n}$.  Then there is a basis whose leading terms are contained in either $\Phir{n} \setminus \Phir{n + 1}$ or $\Phir{n + 1} \cap \Phir{n}$.  For the latter the argument in the proof of \autoref{lemOto} applies and we get that those basis elements are just the $x_\beta$ for $\beta \in \Phir{n + 1} \cap \Phir{n}$.  We must show that the same is true for basis elements whose leading term is in $\Phir{n} \setminus \Phir{n + 1}$.

We have
\begin{align*}
\Phir{n} \setminus \Phir{n + 1} &= \set{\epsilon_i - \epsilon_{n + 1} \ | \ 1 \leq i < n + 1} \\
\Phir{n + 1} \setminus \Phir{n} &= \set{\epsilon_{n + 1} - \epsilon_j \ | \ n + 1 < j \leq 2n + 1}
\end{align*} (where the notation follows Bourbaki~\cite[4.7]{bourbakiLie2})
so the remaining basis elements are of the form
\[b_i = x_{\epsilon_i - \epsilon_{n + 1}} + \sum_{n + 1 < j \leq 2n + 1}a_{ij}x_{\epsilon_{n + 1} - \epsilon_j}\]
for $1 \leq i < n + 1$.  Now we compute
\[[b_i, b_{i'}] = \sum_{n + 1 < j \leq 2n + 1}a_{i'j}N_{\epsilon_i - \epsilon_{n + 1}, \epsilon_{n + 1} - \epsilon_j}x_{\epsilon_i - \epsilon_j} + a_{ij}N_{\epsilon_{n + 1} - \epsilon_j, \epsilon_{i'} - \epsilon_{n + 1}}x_{\epsilon_{i'} - \epsilon_j}.\]
As $n \geq 2$ we may choose $j \neq j'$.  That this expression must equal $0$ consequently gives $a_{ij} = 0$ for all $i$ and $j$.  Thus $b_j = x_{\epsilon_{n + 1} - \epsilon_j}$ and $\cE = \Phir{n + 1}$ as desired.

\subsection{Type $B_n$, $n \geq 5$}

Recall that we define
\[\epsilon_i = \alpha_i + \alpha_{i+1} + \cdots + \alpha_n\]
where the $\alpha_i$ are the simple roots ordered as in Bourbaki~\cite[6, \S 4]{bourbakiLie2}.  Let $\succeq$ be the reverse lexicographical ordering given by
\[\alpha_1 \succ \cdots \succ \alpha_{n - 1} \succ \alpha_n.\]
Specifically, we have
\[\epsilon_r - \epsilon_s \succeq \epsilon_t \succeq \epsilon_i + \epsilon_j\]
for all $i, j, t, r, s$ and if $i < j$ and $r < s$ then
\begin{align*}
\epsilon_r &\succeq \epsilon_i && \text{if} \ i < r, \\
\epsilon_r + \epsilon_s &\succeq \epsilon_i + \epsilon_j && \text{if} \ j < s \ \text{or if} \ j = s \ \text{and} \ i < r, \\
\epsilon_r - \epsilon_s &\succeq \epsilon_i - \epsilon_j && \text{if} \ s < j \ \text{or if} \ j = s \ \text{and} \ i < r.
\end{align*}
One can also compute that if $i < j < n$ then $N_{\epsilon_i + \epsilon_n, \epsilon_j - \epsilon_n} = -N_{\epsilon_j + \epsilon_n, \epsilon_i - \epsilon_n} = 1$

Now if we define
\begin{align*}
R_1 &= \set{\epsilon_i + \epsilon_j \ | \ 1 \leq i < j < n} \\
R_2 &= \set{\epsilon_i + \epsilon_n \ | \ 1 \leq i < n} \\
R_3 &= \set{\epsilon_i - \epsilon_n \ | \ 1 \leq i < n}
\end{align*}
then the set of positive roots, $\Phi^+$, of $B_n$ is the union of the sets
\[\set{\epsilon_i - \epsilon_j \ | \ 1 \leq i < j < n} \succ R_3 \succ \set{\epsilon_i \ | \ 1 \leq i \leq n} \succ R_2 \succ R_1.\]
and
\begin{align*}
S_t &= R_1 \cup R_2 \cup \set{\epsilon_t} \\
S^*_t &= R_1 \cup R_3 \cup \set{\epsilon_t}  
\end{align*}
(see \autoref{not:Bn} for the definition of $S_t, S^*_t$).

One can check that the following subalgebras are elementary and are maximal if not all $a_i$ are zero.
\begin{align*}
B(a_1, \ldots, a_n) &= \spn_k\set{x_\beta, \sum_{i = 1}^na_ix_{\epsilon_i} \ \middle| \ \beta \in R_1 \cup R_2} \\
C(a_1, \ldots, a_{n - 1}) &= \spn_k\set{x_\beta, \sum_{i = 1}^{n - 1}a_ix_{\epsilon_i} \ \middle| \ \beta \in R_1 \cup R_3}.
\end{align*}

\begin{thm} \label{thmB}
If $n \geq 4$ and $\cE \in \max(B_n)$ satisfies $\lt(\cE) = S_t$ or $S^*_t$ then there exist $a_1, \ldots, a_n$ such that $\cE = B(a_1, \ldots, a_n)$ or $C(a_1, \ldots, a_{n - 1})^{\exp(\ad(a_nx_{\alpha_n}))}$ respectively.
\end{thm}
\begin{proof}
If $\lt(\cE) = S_t$ for some $t$ then the argument of \autoref{lemOto} immediately gives $\cE = B(0, \ldots, 0, 1, a_{t + 1}, \ldots, a_n)$ for some $a_{t + 1}, \ldots, a_n$.  Now assume that $\lt(\cE) = S^*_t$ for some $1 \leq t < n$.  The reduced echelon form basis of $\cE$ then consists of the elements $x_{\epsilon_i + \epsilon_j}$ where $1 \leq i < j < n$ and for $1 \leq i < n$ the elements
\[x = x_{\epsilon_t} + \sum_{s = 1}^{t - 1}a_sx_{\epsilon_s} + \sum_{s = 1}^{n - 1}b_sx_{\epsilon_s + \epsilon_n} \quad \text{and} \quad y_i = x_{\epsilon_i - \epsilon_n} + \sum_{s = 1}^nc_{is}x_{\epsilon_s} + \sum_{s = 1}^{n - 1}d_{is}x_{\epsilon_s + \epsilon_n}\]
for some $a_s$, $b_s$, $c_{sk}$, and $d_{sk}$.  That it's reduced means $c_{it} = 0$ for all $i$.  Notice that $\exp(\ad(\lambda x_{\alpha_n}))$ is upper triangular with respect to $\preceq$ so $\lt(\exp(\ad(\lambda x_{\alpha_n}))(\cE)) = \lt(\cE)$ and the element in the reduced basis of $\exp(\ad(\lambda x_{\alpha_n}))(\cE)$ with leading term $x_{\epsilon_t}$ is $\exp(\ad(\lambda x_{\alpha_n}))(x)$.  When $\lambda = -b_tN_{\epsilon_n, \epsilon_t}^{-1}$ we have
\[\exp(\ad(\lambda x_{\alpha_n}))(x) = x_{\epsilon_t} + \sum_{s = 1}^{t - 1}a_sx_{\epsilon_s} + \sum_{s = 1}^{t - 1}(b_s - a_sb_tN_{\epsilon_n, \epsilon_t}^{-1}N_{\epsilon_n, \epsilon_s})x_{\epsilon_s + \epsilon_n} + \sum_{s = t + 1}^{n - 1}b_sx_{\epsilon_s + \epsilon_n}.\]
As $x_{\epsilon_t + \epsilon_n}$ is not a term in this basis element it suffices to show that $\cE$ is the subalgebra $C(0, \ldots, 0, 1, a_{t + 1}, \ldots, a_{n - 1})$ when $b_t = 0$.  We do this by showing that the $b_s$, $c_{ij}$, and $d_{ij}$ must all be zero.

One can check that the coefficient of $x_{\epsilon_i}$ in $[y_i, y_j]$ is $N_{\epsilon_i - \epsilon_n, \epsilon_n}c_{jn}$ so $c_{jn} = 0$ for all $j$.  Also if $j \neq t$ then the coefficient of $x_{\epsilon_j + \epsilon_t}$ in $[x, y_i]$ is $N_{\epsilon_t, \epsilon_j}c_{ij}$ so $c_{ij} = 0$ for all $i, j$.  For $i \neq t$ the coefficient of $x_{\epsilon_i + \epsilon_t}$ in $[x, y_t]$ is $N_{\epsilon_i + \epsilon_n, \epsilon_t - \epsilon_n}b_i$ so $b_i = 0$ for all $i$.

If $i, j, t < n$ are distinct then the coefficients of $x_{\epsilon_j + \epsilon_t}$ and $x_{\epsilon_i + \epsilon_j}$ in $[y_i, y_t]$ are $N_{\epsilon_j + \epsilon_n, \epsilon_t - \epsilon_n}d_{ij}$ and $N_{\epsilon_i - \epsilon_n, \epsilon_j + \epsilon_n}d_{tj}$ respectively.  As $n$ is at least $4$ this gives $d_{ij} = 0$ for all $i \neq j$.

Finally, if $i, j < n$ are distinct then $N_{\epsilon_i - \epsilon_n, \epsilon_j + \epsilon_n}d_{jj} + N_{\epsilon_i + \epsilon_n, \epsilon_j - \epsilon_n}d_{ii}$ is the coefficient of $x_{\epsilon_i + \epsilon_j}$ in $[y_i, y_j]$.  For $i < j < t < n$ we thus get a system of equations
\begin{align*}
N_{\epsilon_i - \epsilon_n, \epsilon_j + \epsilon_n}d_{jj} + N_{\epsilon_i + \epsilon_n, \epsilon_j - \epsilon_n}d_{ii} &= d_{jj} + d_{ii} = 0 \\
N_{\epsilon_i - \epsilon_n, \epsilon_t + \epsilon_n}d_{tt} + N_{\epsilon_i + \epsilon_n, \epsilon_t - \epsilon_n}d_{ii} &= d_{tt} + d_{ii} = 0 \\
N_{\epsilon_j - \epsilon_n, \epsilon_t + \epsilon_n}d_{tt} + N_{\epsilon_j + \epsilon_n, \epsilon_t - \epsilon_n}d_{jj} &= d_{tt} + d_{jj} = 0
\end{align*}
whose unique solution is $d_{ii} = d_{jj} = d_{tt} = 0$.  This gives $d_{ii} = 0$ for all $i$ and completes the proof of the theorem.
\end{proof}

We note that $\mathbb E(\mathfrak u) = \lie(\max(\Phi))$ does not hold in type $B_n$, $n \geq 5$.  In the next proposition we show that any elementary subalgebra in $\bE(\fu)$ is $G$-conjugate to a subalgebra in $\lie(\max(\Phi))$. 
\begin{prop} 
\label{Prop:Bconj} Let $F/\bF_p$ be a field extension. 
Any $F$-point of $\mathbb E(\fu)$ is $G(F)$-conjugate to an elementary subalgebra in $\lie(\max(\Phi))$. 
\end{prop} 

\begin{proof} We show that any elementary subalgebra in $\bE(\fu)$ defined over  $F$ is $G(F)$-conjugate to $\lie(S_1)$ where $S_1$ is as defined in \autoref{not:Bn}. 

The simple reflection $s_n$ acts by negating $\epsilon_n$ and fixing the remaining $\epsilon_i$ therefore any representative $\dot{s_n} \in N_G(T)$ conjugates 
$C(a_1, \ldots, a_{n - 1})$ to $B(a_1, \ldots, a_{n - 1}, 0)$.  Similarly $s_i$, where $i < n$, swaps $\epsilon_i$ with $\epsilon_{i + 1}$ and fixes the remaining $\epsilon_j$ so by conjugation we may assume our elementary subalgebra is of the form $B(a_1, \ldots, a_{n - 1}, 1)$.  Finally conjugation by $\exp(\ad(a_iN_{\epsilon_i - \epsilon_n, \epsilon_n}^{-1}x_{\epsilon_i - \epsilon_n}))$ lets us assume $a_i = 0$ and does not alter the remaining $a_j$, thus we have conjugated our subalgebra to $B(0, \ldots, 0, 1)$.  Using simple reflections we conjugate to $B(1, 0, \ldots, 0) = \lie(S_1)$ and are done. Note that if the variables $(a_1, \ldots, a_{n-1}, a_n)$ belong to the field $F$ then all conjugations we have to perform to reduce $C(a_1, \ldots, a_{n - 1})$ or $B(a_1, \ldots, a_{n - 1}, a_n)$ to $B(1, 0, \ldots, 0)$ are by elements in $G(F)$. 
\end{proof} 

\subsection{Type $B_4$}

We keep the ordering and choice of basis from the last section.  As \autoref{thmB} applies here as well, all that is left is to prove the following.

\begin{thm}
\label{thm:B4}
If $\lt(\cE) = \Phir{1}$ then $\cE = \lie(\Phir{1})$.
\end{thm}
\begin{proof}
The reduced echelon form basis of $\cE$ is of the form
\begin{align*}
x_i &= x_{\epsilon_1 + \epsilon_i} + \sum_{\substack{3 \leq r \leq i \\ 2 \leq s < r}}c_{isr}x_{\epsilon_s + \epsilon_r}, \quad \text{and} \\
y_j &= x_{\epsilon_1 - \epsilon_j} + \sum_{\substack{j \leq r \leq 4 \\ 2 \leq s < r}}d_{jsr}x_{\epsilon_s - \epsilon_r} + \sum_{\substack{r = 3, 4 \\ 2 \leq s < j}}e_{jsr}x_{\epsilon_s + \epsilon_r}
\end{align*}
for $i = 2, 3, 4$ and $j = 1, 2, 3, 4$.  These $7$ basis elements can be formed into $21$ possible commutators which must equal zero.  Setting their coefficients equal to zero gives a system of equations which can be solved by hand and whose unique solution is $c_{isr} = d_{jsr} = e_{jsr} = 0$ as desired.
\end{proof}

As above we do not get $\mathbb E(\mathfrak u) = \lie(\max(\Phi))$.  Instead we get $\bE(\fu) \subseteq G\cdot\lie(S_1) \cup \set{\lie(\Phir{1})}$.

\subsection{Type $D_n$, $n \geq 5$}

Define
\[\epsilon_i = \begin{cases} \alpha_i + \alpha_{i + 1} + \cdots + \alpha_{n - 2} + \frac12(\alpha_{n - 1} + \alpha_n) & \text{if} \ i \leq n - 2 \\ \frac12(\alpha_{n - 1} + \alpha_n) & \text{if} \ i = n - 1 \\  \frac12(\alpha_n - \alpha_{n - 1}) & \text{if} \ i = n. \end{cases}\]
where the $\alpha_i$ are the simple roots ordered as in Bourbaki~\cite[6, \S 4]{bourbakiLie2} and let $\succeq$ be the reverse lexicographical ordering given by
\[\alpha_{n - 1} \succ \cdots \succ \alpha_2 \succ \alpha_1 \succ \alpha_{n - 1} \succ \alpha_n.\]
One can check that for $i < j < n$ this gives $N_{\epsilon_i - \epsilon_n, \epsilon_j + \epsilon_n} = -N_{\epsilon_j - \epsilon_n, \epsilon_i + \epsilon_n} = 1$.

If we define
\[R = \set{\epsilon_i + \epsilon_j \ | \ 1 \leq i < j \leq n - 1}\]
then
\begin{align*}
\Phir{n} &= R \cup \set{\epsilon_i + \epsilon_n \ | \ 1 \leq i \leq n - 1} \\
\Phir{n - 1} &= R \cup \set{\epsilon_i - \epsilon_n \ | \ 1 \leq i \leq n - 1}
\end{align*}
and $\Phi^+$ is the union of the sets
\[\Phi^+ \setminus \Phir{\set{\alpha_1, \alpha_{n - 1}, \alpha_n}} \succ \Phir{1} \setminus \Phir{\set{\alpha_{n - 1}, \alpha_n}} \succ \Phir{n - 1} \setminus R \succ \Phir{n}\setminus R \succ R.\]

\begin{thm} \label{thmD}
If $n \geq 4$ and $\cE \in \max(D_n)$ satisfies $\lt(\cE) = \Phir{n}$ or $\Phir{n - 1}$ then $\cE = \lie(\Phir{n})$ or $\lie(\Phir{n - 1})$ respectively.
\end{thm}
\begin{proof}
If $\lt(\cE) =\Phir{n}$ then \autoref{lemOto} immediately gives $\cE = \lie(\Phir{n})$, so assume $\lt(\cE) = \Phir{n - 1}$.  The reduced basis of $\cE$ then consists of the Chevalley basis elements corresponding to roots in $R$ and the elements
\[y_i = x_{\epsilon_i - \epsilon_n} + \sum_{s = 1}^{n - 1}a_{is}x_{\epsilon_s + \epsilon_n}\]
for $1 \leq i \leq n - 1$ and we want to show that $a_{ij} = 0$ for all $i, j$.

If $i, j, t < n$ are distinct then the coefficient of $x_{\epsilon_i + \epsilon_t}$ in $[y_i, y_j]$ is $N_{\epsilon_i - \epsilon_n, \epsilon_t + \epsilon_n}a_{jt}$.  As $n \geq 4$ this gives $a_{jt} = 0$ for all $j \neq t$.  Now for $i \neq j$ the coefficient of $x_{\epsilon_i + \epsilon_j}$ in $[y_i, y_j]$ is $N_{\epsilon_i - \epsilon_n, \epsilon_j + \epsilon_n}a_{jj} + N_{\epsilon_i + \epsilon_n, \epsilon_j - \epsilon_n}a_{ii}$.  Thus if $i < j < t < n$ we get a system of equations
\begin{align*}
N_{\epsilon_i - \epsilon_n, \epsilon_j + \epsilon_n}a_{jj} + N_{\epsilon_i + \epsilon_n, \epsilon_j - \epsilon_n}a_{ii} &= a_{jj} + a_{ii} = 0 \\
N_{\epsilon_i - \epsilon_n, \epsilon_t + \epsilon_n}a_{tt} + N_{\epsilon_i + \epsilon_n, \epsilon_t - \epsilon_n}a_{ii} &= a_{tt} + a_{ii} = 0 \\
N_{\epsilon_j - \epsilon_n, \epsilon_t + \epsilon_n}a_{tt} + N_{\epsilon_j + \epsilon_n, \epsilon_t - \epsilon_n}a_{jj} &= a_{tt} + a_{jj} = 0
\end{align*}
whose unique solution is $a_{ii} = a_{jj} = a_{tt} = 0$.  This gives $a_{ii} = 0$ for all $i$ and completes the proof of the theorem.
\end{proof}

\subsection{Type $D_4$}

We keep the ordering from the last section.  As \autoref{thmD} applies here as well, all that is left is to prove that if $\lt(\cE) = \Phir{1}$ then $\cE = \lie(\Phir{1})$.  The reduced basis of $\cE$ is of the form
\begin{align*}
& x_{\epsilon_1 - \epsilon_2} + a_{11}x_{\epsilon_2 - \epsilon_3} + a_{12}x_{\epsilon_3 - \epsilon_4} + a_{13}x_{\epsilon_2 - \epsilon_4} + a_{14}x_{\epsilon_3 + \epsilon_4} + a_{15}x_{\epsilon_2 + \epsilon_4} + a_{16}x_{\epsilon_2 + \epsilon_3} \\
& x_{\epsilon_1 - \epsilon_3} + a_{21}x_{\epsilon_3 - \epsilon_4} + a_{22}x_{\epsilon_2 - \epsilon_4} + a_{23}x_{\epsilon_3 + \epsilon_4} + a_{24}x_{\epsilon_2 + \epsilon_4} + a_{25}x_{\epsilon_2 + \epsilon_3} \\
& x_{\epsilon_1 - \epsilon_4} + a_{31}x_{\epsilon_3 + \epsilon_4} + a_{32}x_{\epsilon_2 + \epsilon_4} + a_{33}x_{\epsilon_2 + \epsilon_3} \\
& x_{\epsilon_1 + \epsilon_4} + a_{41}x_{\epsilon_2 + \epsilon_3} \\
& x_{\epsilon_1 + \epsilon_3} \\
& x_{\epsilon_1 + \epsilon_2} \\
\end{align*}
and we wish to show that the $a_{ij}$ are zero.  As in the $B_4$ case, setting commutators equal to zero yields a system of equations for the $a_{ij}$ which can be solved by hand and whose unique solution is $a_{ij} = 0$ for all $i, j$.

\subsection{Type $G_2$} \label{sec:G2unip}

As with type $A_2$ this case is exceptional in that it is not true that every elementary subalgebra is conjugate to a subalgebra in $\lie(\max(\Phi))$ and it is not true that every Weyl group orbit in $\max(\Phi)$ contains an ideal.

Recall that we assume the characteristic $p$ is good for $\Phi$, so explicitly we assume $p \neq 2, 3$.  Let $\alpha_1$ be the short root and $\alpha_2$ the long root in the basis, with $s_1, s_2 \in W$ the corresponding simple reflections.  We choose the reverse graded lexicographic ordering on $\Phi$ with $\alpha_1 \succ \alpha_2$ so that the positive roots are ordered as follows:
\[\alpha_1 \succ \alpha_2 \succ \alpha_1 + \alpha_2 \succ 2\alpha_2 + \alpha_2 \succ 3\alpha_1 + \alpha_2 \succ 3\alpha_1 + 2\alpha_2.\]
The $s_i$ act via
\begin{align*}
s_i(\alpha_i) &= -\alpha_i, \\
s_1(\alpha_2) &= 3\alpha_1 + \alpha_2, \\
s_2(\alpha_1) &= \alpha_1 + \alpha_2,
\end{align*}
and $W$ is the dihedral group of order $12$ generated by the reflection $s_1$ and rotation $s_1s_2$.  It is simple to check that the maximal sets of commuting positive roots fall into two orbits:
\begin{align*}
C_1 &= \set{\alpha_1, 3\alpha_1 + \alpha_2, 3\alpha_1 + 2\alpha_2}, \\
C_2 &= \set{\alpha_1 + \alpha_2, 3\alpha_1 + \alpha_2, 3\alpha_1 + 2\alpha_2}, \\
C_3 &= \set{\alpha_2, 2\alpha_1 + \alpha_2, 3\alpha_1 + 2\alpha_2},
\end{align*}
and
\begin{align*}
C_4 &= \set{\alpha_2, \alpha_1 + \alpha_2, 3\alpha_1 + 2\alpha_2}, \\
C_5 &= \set{2\alpha_1 + \alpha_2, 3\alpha_1 + \alpha_2, 3\alpha_1 + 2\alpha_2}.
\end{align*}
Only $C_5$ is an ideal.
\begin{thm}
Every $\cE \in \mathbb E(\mathfrak u)$ is $G$-conjugate to one of $\lie(C_3)$, $\lie(C_5)$, or $L = \langle x_{\alpha_2} + x_{3\alpha_1 + \alpha_2}, x_{2\alpha_1 + \alpha_2}, x_{3\alpha_1 + 2\alpha_2}\rangle$.  Moreover, these subalgebras are pairwise non-conjugate.
\end{thm}
\begin{proof}
As $C_3$ and $C_5$ are representatives of the two orbits we must show that every $\cE \in \mathbb E(\mathfrak g)$ is conjugate to either $L$ or $\lie(C_i)$ for any $i$.  We handle each of the five possible leading terms separately.

An abelian subalgebra with leading terms $C_1$ is generated by elements of the form
\begin{align*}
& x_{\alpha_1} + a_1x_{\alpha_1 + \alpha_2} + a_2x_{2\alpha_1 + \alpha_3}, \\
& x_{3\alpha_1 + \alpha_2}, \\
& x_{3\alpha_1 + 2\alpha_2}.
\end{align*}
for some $a_1, a_2 \in k$.  Conjugating by $\exp(\ad(\frac12a_2x_{\alpha_1 + \alpha_2}))\exp(\ad(a_1x_{\alpha_2}))$ gives $\lie(C_1)$.

An abelian subalgebra with leading terms $C_2$ is generated by elements of the form
\begin{align*}
& x_{\alpha_1 + \alpha_2} + a_1x_{2\alpha_1 + \alpha_2}, \\
& x_{3\alpha_1 + \alpha_2}, \\
& x_{3\alpha_1 + 2\alpha_2},
\end{align*}
for some $a_1 \in k$.  Conjugating by $\exp(\ad(-\frac12a_1x_{\alpha_1}))$ gives $\lie(C_2)$.

An abelian subalgebra with leading terms $C_3$ is generated by elements of the form
\begin{align*}
& x_{\alpha_2} + a_1x_{\alpha_1 + \alpha_2} + a_2x_{3\alpha_1 + \alpha_2}, \\
& x_{2\alpha_1 + \alpha_2} + 3a_1x_{3\alpha_1 + \alpha_2}, \\
& x_{3\alpha_1 + 2\alpha_2}.
\end{align*}
for some $a_1, a_2 \in k$.  Conjugation by $\exp(\ad(-a_1x_{\alpha_1}))$ allows us to assume $a_1 = 0$, giving
\begin{align*}
& x_{\alpha_2} + a_2x_{3\alpha_1 + \alpha_2}, \\
& x_{2\alpha_1 + \alpha_2}, \\
& x_{3\alpha_1 + 2\alpha_2}.
\end{align*}
If $a_2 = 0$ this is $\lie(C_3)$.  If $a_2 \neq 0$ then conjugation by $\alpha_2^\vee(\sqrt[3]{a_2})$ sends the subalgebra above to $L$.

The only subalgebra with leading terms $C_5$ is $\lie(C_5)$ so all that is left is $C_4$.  An abelian subalgebra with leading terms $C_4$ is generated by elements of the form
\begin{align*}
& x_{\alpha_2} + a_1x_{2\alpha_1 + \alpha_2} + a_2x_{3\alpha_1 + \alpha_2}, \\
& x_{\alpha_1 + \alpha_2} + a_3x_{2\alpha_1 + \alpha_2} - 3a_1x_{3\alpha_1 + \alpha_2}, \\
& x_{3\alpha_1 + 2\alpha_2}.
\end{align*}
for some $a_1, a_2, a_3 \in k$.  Conjugation by $\exp(\ad(-\frac12a_3x_{\alpha_1}))$ allows us to assume $a_3 = 0$.  If $a_1 = 0$ then conjugation by $\dot{s_1}$ yeilds a subalgebra generated by elements of the form
\begin{align*}
& ax_{\alpha_2} + x_{3\alpha_1 + \alpha_2}, \\
& x_{2\alpha_1 + \alpha_2}, \\
& x_{3\alpha_1 + 2\alpha_2},
\end{align*}
and the leading terms are either $C_3$ or $C_5$, so assume $a_1 \neq 0$.  Then conjugation by $\alpha_2^\vee(\sqrt[3]{a_1})$ allows us to assume $a_1 = 1$.  Our generators are now of the form
\begin{align*}
& x_{\alpha_2} + x_{2\alpha_1 + \alpha_2} + ax_{3\alpha_1 + \alpha_2}, \\
& x_{\alpha_1 + \alpha_2} - 3x_{3\alpha_1 + \alpha_2}, \\
& x_{3\alpha_1 + 2\alpha_2}.
\end{align*}
For any choice of $u, v \in k$ we may conjugate by $\exp(\ad(ux_{\alpha_1}))\dot{s_1}\exp(\ad(vx_{\alpha_1}))$ to get a subalgebra generated by terms of the form
\begin{align*}
& \lambda_1x_{\alpha_2} + ((u^3 + 3u + a)v + u^2 + 1)x_{\alpha_1 + \alpha_2} + \ \text{lower terms}\ldots, \\
& \lambda_2x_{\alpha_2} + (3(u^2 - 1)v + 2u)x_{\alpha_1 + \alpha_2} + \text{lower terms}\ldots, \\
& x_{3\alpha_1 + 2\alpha_2}.
\end{align*}
From the constant term one sees that for any $a$ the polynomial $u^4 - 6u^2 - 2ua - 3$ must always have a root distinct from $\pm1$ and necessarily nonzero.  Choose $u$ to be such a root.  Then
\[\frac{4u + a}{u^2 - 1} = \frac{u^2 + 3}{2u}\]
holds and, moreover, implies that the equations
\begin{align*}
3(u^2 - 1)v + 2u &= 0, \\
3(4u + a)v + u^2 + 3 &= 0
\end{align*}
determine a unique $v$.  Finally
\[u[3(u^2 - 1)v + 2u] + [3(4u + a)v + u^2 + 3] = 3[(u^3 + 3u + a)v + u^2 + 1]\]
so $(u^3 + 3u + a)v + u^2 + 1 = 0$.  Thus with this choice of $u, v$ we have conjugated the subgroup to one for whom $\alpha_1 + \alpha_2$ is not a leading term.  This means we are no longer in the case of leading terms $C_4$ and our previous arguments apply.  Hence we have shown that every $\cE \in \mathbb E(\mathfrak u)$ is $G$-conjugate to one of $\lie(C_3)$, $\lie(C_5)$, or $L$.

All that is left is to show that these three subalgebras are not conjugate.  For this we simply observe that their normalizers
\begin{align*}
N_{\mathfrak g}(\lie(C_3)) &= \langle h_1, h_2, x_\beta \ | \ \beta \in \Phi^+ \setminus \set{\alpha_1}\rangle, \\
N_{\mathfrak g}(\lie(C_5)) &= \langle h_1, h_2, x_\beta \ | \ \beta \in \Phi^+ \cup \set{-\alpha_2}\rangle, \\
N_{\mathfrak g}(L) &= \langle h_1 + 2h_2, x_\beta \ | \ \beta \in \Phi^+ \setminus \set{\alpha_1}\rangle,
\end{align*}
have dimensions $7$, $9$, and $6$ respectively.  Conjugate subalgebras have conjugate, hence equidimensional, normalizers therefore the subalgebras $\lie(C_3)$, $\lie(C_5)$, and $L$ are non-conjugate.
\end{proof}

\subsection{Types $E_6$, $E_8$, and $F_4$}

For these types there are too many cases for us to reasonably tackle them by hand.  Instead we have written Magma code that attempts to confirm that $\bE(\fu) \subseteq U\cdot\lie(\max(\Phi))$ holds for Lie algebras of a given type.  This code is available online~\cite{starkRootsMgm}, is successful for types $E_6$, $E_8$, and $F_4$, and moreover confirms that if $\cE \in \bE(\fu)$ is defined over a subfield $F \subseteq k$ then the conjugating element may be taken from the $F$-points of $U$.  We warn that while the computations for $E_6$ and $F_4$ are very quick the computation for $E_8$ takes several hours and it is helpful to import the list of maximal sets of commuting roots from code written in Sage~\cite{starkCommRoots}, where that aspect of the computation is much faster.

Summarizing the above discussion, we now have the following result.

\begin{thm}
Let $G$ be simple algebraic group with root system $\Phi$, not of type $A_2$ or $G_2$.  Then
\[\bE(\fu) = \bigcup_{I \in \max(\Phi)}U\cdot\lie(I).\]
Moreover, for any subfield $F \subseteq k$ the $F$-points of $\bE(\fu)$ are contained in $U(F)\cdot\lie(I)$ for some $I \in \max(\Phi)$.
\end{thm}

If $\dot w \in N_G(T)$ is a representative of the Weyl group element $w \in W$ then $\dot w \cdot x_\alpha$ is proportional to $x_{w\cdot\alpha}$.  Thus if both $I$ and $wI$ are contained in $\max(\Phi)$ then $\dot w\cdot\lie(I) = \lie(wI)$.  As $\Phi$ is not of type $G_2$ we know that every set in $\max(\Phi)$ is $W$-conjugate to an ideal.  This gives the following corollary.

\begin{cor} \label{thm:EuIdeal}
Let $G$ be simple algebraic group with root system $\Phi$, not of type $A_2$ or $G_2$.  Then
\[\bE(\fu) \subseteq \bigcup_{\substack{I \in \max(\Phi) \\ I \ \mathrm{an \ ideal}}}G\cdot\lie(I).\]
Moreover, for any subfield $F \subseteq k$ the $F$-points of $\bE(\fu)$ are contained in $G(F)\cdot\lie(I)$ for some ideal $I \in \max(\Phi)$.
\end{cor}

\subsection{Type $G_2$ when $p = 3$} \label{sec:G2p3}

We now break with the assumption, made at the start of the current section, that $p$ is equal or greater than the length of the longest root string in $\Phi$.  In type $G_2$ there are root strings of length $4$ and we now consider the case $p = 3$.

The results of this section are based on the property that if $\alpha \neq \beta$ are positive roots then $[x_\alpha, x_\beta] = 0$ if and only if $\alpha$ and $\beta$ commute.  The techniques that follow from that assumption are still valid in the current case if we replace ``commuting'' with a different combinatorial property that holds if and only if the corresponding elements of the Chevalley basis commute.

\begin{defn}
Let $\alpha \neq \beta$ be positive roots and $p$ a prime.  We say that $\alpha$ and $\beta$ \emph{$p$-commute} if they commute or if the $\alpha$-string through $\beta$ begins at $(-p + 1)\alpha + \beta$.
\end{defn}

The above notion is a direct translation of the condition that the structure constant $N_{\alpha, \beta}$ is either $0$ or $p$ (it cannot be a larger multiple of $p$) therefore the product $[x_\alpha, x_\beta] = 0$ if and only if $\alpha$ and $\beta$ $p$-commute.  There are $3$ sets of maximal $p$-commuting roots in $\Phi^+$, they are
\begin{align*}
R_1 &= \set{\alpha_1 + \alpha_2, 2\alpha_1 + \alpha_2, 3\alpha_1 + \alpha_2, 3\alpha_1 + 2\alpha_2}, \\
R_2 &= \set{\alpha_1, 2\alpha_1 + \alpha_2, 3\alpha_1 + \alpha_2, 3\alpha_1 + 2\alpha_2}, \\
R_3 &= \set{\alpha_2, \alpha_1 + \alpha_2, 2\alpha_1 + \alpha_2, 3\alpha_1 + 2\alpha_2}. \\
\end{align*}
Geometrically each set consists of $4$ rotationally consecutive roots.  As long and short roots alternate and the Weyl group is the Dihedral group on the long roots we see that these three sets are conjugate.

The three cases of elementary abelian Lie subalgebras that come from these roots are as follows.  If $\lt(\cE) = R_1$ then $\cE = \Lie(R_1)$ because $R_1$ is minimal in the ordering.  If $\lt(\cE) = R_2$ then
\[\cE = \langle x_{\alpha_1} + cx_{\alpha_1 + \alpha_2}, x_{2\alpha_1 + \alpha_2}, x_{3\alpha_1 + \alpha_2}, x_{3\alpha_1 + 2\alpha_2}\rangle.\]
Conjugating by $\exp(\ad(-cx_{\alpha_2}))$ yields $\Lie(R_2)$.  Lastly if $\lt(\cE) = R_2$ then
\[\cE = \langle x_{\alpha_2} + cx_{3\alpha_1 + \alpha_2}, x_{\alpha_1 + \alpha_2}, x_{2\alpha_1 + \alpha_2}, x_{3\alpha_1 + 2\alpha_2}\rangle.\]
Conjugating by $\exp(\ad(\sqrt[3]{c}x_{\alpha_2}))$ yields $\Lie(R_3)$.  Thus we have $\bE(\fu) \subseteq G\cdot\Lie(R_1)$.


\section{Calculation of $\bE(\fg)$ for $\fg = \Lie G$}
\label{sec:reduc}

We now calculate, when $p$ is separably good, the variety $\bE(\fg)$ when $\fg$ is the Lie algebra of a reductive group $G$.  We start by reducing to the case of a simple algebraic group.

\begin{thm}
Let $G_1, \ldots, G_n$ be the simple algebraic subgroups of the derived group $[G, G]$ and let $\fg_i = \lie(G_i)$.  Then
\[\bE(\fg) = \prod_{i = 1}^n\bE(\fg_i).\]
\end{thm}
\begin{proof}
Follows from \autoref{thm:Eprod}.
\end{proof}

Henceforth we assume that $G$ is a simple algebraic group.  In all but the $A_2$ and $G_2$ cases we find that $\bE(\fg)$ is a disjoint union of flag varieties corresponding to ideals of maximal commuting roots.  In the $A_2$ and $G_2$ cases we find that the $\bE(\fg)$ are irreducible varieties of dimensions $5$ and $8$ respectively, each given as the union of three  $G$-orbits.

\subsection{Reduction to the unipotent case}

Our first step is to show that every maximal elementary abelian may be conjugated into $\fu$, the Lie algebra of the unipotent radical of the Borel.  This is a result of Levy et.\ al.\ based on the following theorem.

\begin{thm}[{\cite[2.2]{lmtNilSubalgebras}}]
Let $G$ be a semisimple group and $p$ a very non-torsion prime for $G$.  Let $U$ be a unipotent subgroup scheme of $G$.  Then $U$ is contained in a Borel subgroup of $G$.
\end{thm}

\begin{remark}
Reading the proof in Levy et al., one finds that the theorem above holds for any prime when $G$ is $\SL_n$ or $\Sp_n$,
and consequently it holds for any group which is separably isogeneous to $\SL_n$ or $\Sp_n$. 
\end{remark}

\begin{cor} \label{cor:conj}
Let $G$ be semisimple and $p$ either separably good or very non-torsion for $G$.  Then any elementary subalgebra $\cE$ of $\fg$ may be conjugated into $\fu$.
\end{cor}
\begin{proof}
Per remark (a) following the above theorem (\emph{loc.\ cit.})\ observe that the first Frobenius kernel $G_1 \leq G$ is a closed normal subgroup isomorphic to the infinitesimal group scheme $\underline\fg$ that one obtains from the restricted Lie algebra $\fg$~\cite[I.9.6]{jantzen} and the conjugation action $G \curvearrowright G_1$ corresponds to the adjoint action $G \curvearrowright \fg$.  Let $\underline\cE \subseteq \underline\fg$ be the subgroup scheme corresponding to an elementary subalgebra.  As $\cE$ consists entirely of nilpotent elements of $\mathfrak g$ we have, by Engel's theorem, that $\underline\cE$ is a unipotent subgroup scheme of $G$.  The above theorem then gives that $\underline\cE$ may be conjugated into our choice of Borel $B$ and therefore $\cE$ may be conjugated into the Lie algebra of $B$.  As $\cE$ is nilpotent its conjugate must then lie in the nilpotent radical $\mathfrak u$ of this Lie algebra.
\end{proof}

\begin{remark}  Explicitly, the condition on $p$ in \autoref{cor:conj} amounts to $p$ separably good or $G$ of type $G_2$ and $p = 3$. The latter produces an interesting counter-example to the separability theorem \autoref{thm:sep}  which we explain in \autoref{ex:G2}. 
\end{remark}

\begin{example} \label{ex:pgl2}
This conjugation property is not true for a general prime.  One can check that when $p = 2$ the nilpotent cone of $\mathfrak{pgl}_2$ is a linear subspace of dimension $2$.  This is necessarily maximal among elementary subalgebras and is not contained in the Lie algebra of any Borel.  Examples of maximal elementary subalgebras which are not contained in a Borel exist for $\mathfrak{pgl}_3$ when $p = 3$ (see Levy et al.~\cite{lmtNilSubalgebras}) and $\mathfrak{pgl}_4$ when $p = 2$.  While it is true that for all $n$ there exists an elementary subalgebra of $\mathfrak{pgl}_n$ which is not contained in a Borel, it is not known for larger $n$ whether such an elementary subalgebra can be maximal.
\end{example}

\begin{cor} \label{thm:EgUnion}
Let $G$ be semisimple and $p$ separably good for $G$.  Then
\[\bE(\fg) = \bigcup_{\substack{R \in \max(\Phi) \\ R \ \mathrm{an \ ideal}}}G\cdot\lie(R).\]
\end{cor}
\begin{proof}
Follows immediately from \autoref{thm:EuIdeal} and the corollary above.
\end{proof}

\subsection{Ideal orbits}

Now we geometrically identify orbits of the form $G\cdot\lie(R)$ where $R \in \max(\Phi)$.  Observe that such $\cE = \lie(R)$ are fixed by $B$ under the adjoint action (see Malle and Testerman~\cite[15.4]{malletesterman}).  The stabilizer $P = \Stab_G(\cE)$ is then a standard parabolic and is generated by $B$ and representatives $\dot{s_i} \in N_G(T)$ of some collection of simple roots $s_i$ in the Weyl group $W = N_G(T)/C_G(T)$.  As $\dot s_i\lie(R) = \lie(s_iR)$ we have $\dot{s_i} \in P$ if and only if $s_i \in \Stab_W(R)$.  Thus $P$ is identified by the information in \autoref{table:stab} above.

The orbit map $\pi\colon G \to G\cdot\cE$ factors to a bijective morphism $G/P \to G\cdot\cE$ and we show below that this morphism is in fact an isomorphism. Note that the fact that the orbit map $\pi$ induces an isomorphism $G/P \cong G\cdot\cE$ is equivalent to $\pi$ being separable. This, in turn, is equivalent to the condition that the kernel of the tangent map at the identity $\ker d\pi_1$  is contained in $\lie(P)$ (see, for example, \cite[6.7]{borelLAG}).  We show that the latter holds in \autoref{thm:sep2}.

\begin{lemma}
\label{thm:sep}
If $\pi\colon G \to \bE(r, \fg)$ is the orbit map $g \mapsto g\cdot\mathcal E$ for some elementary subalgebra $\mathcal E \in \bE(r, \fg)$ then $\ker\mathrm{d}\pi_1 = N_{\fg}(\mathcal E)$.
\end{lemma}
\begin{proof}
Let $\underline e = (e_1, \ldots, e_r)$ be a basis of $\cE$.  As $\bE(r,\fg)$ is a closed subvariety in $\Grass(r, \fg)$ we may consider the latter to be the codomain of $\pi$.  This gives the following diagram
\begin{equation}
\label{eq:sep3}
\xymatrix@R=40pt{& & (\fg^{\times r})^\circ \ar[d]^\phi  \\ G \ar[urr]^-{\wt \pi}_-{g \mapsto g \cdot \underline e} \ar[rr]^-\pi_-{g \mapsto g\cdot\cE} && \Grass(r, \fg)}
\end{equation}
where $(\fg^{\times r})^\circ$ is the open subset of $\fg^{\times r}$ consisting of linearly independent $r$-tuples and $\phi\colon(\fg^{\times r})^\circ \to \Grass(r, \fg)$ is the canonical projection.  The map $\phi$ is, in particular, a $\GL_r$-torsor and hence is locally trivial.  Its tangent map at $\underline e$ can be identified as the linear map
\[\mathrm{d}\phi_{\underline e}: \Hom(\cE, \fg) \to \Hom(\cE, \fg/\cE)\]
induced by the projection $\fg \to \fg/\cE$.  Indeed, locally on the affine neighborhood defined by the non-vanishing of the Plucker coordinate associated with $\cE$, the torsor trivializes as follows:
\begin{equation*}
\label{eq:affine}
\xymatrix{\bA^{r(n-r)} \times \GL_r \ar[d]^-\phi \ar@{^(->}[r] & \bA^{rn} = \bM_{r\times n} \simeq \Hom_k(\cE, \fg)\ar[d] \\ \bA^{r(n-r)} \ar@{=}[r] & \Hom(\cE, \fg/\cE).}
\end{equation*}

Replace \hyperref[eq:sep3]{diagram~\ref*{eq:sep3}} with the corresponding diagram for tangent spaces:

\begin{equation*}
\label{eq:sep3tan}
\xymatrix@R=40pt{& & \fg^{\times r} \cong \Hom(\cE, \fg) \ar[d]^-{\mathrm{d}\phi_{\underline e}} \\ \fg \ar[urr]_-{\mathrm{d}\wt \pi_1}^-{g \mapsto ([g, \underline e])} \ar[rr]_-{\mathrm{d}\pi_1}&& \Hom(\cE, \fg/\cE).}
\end{equation*}
We are interested in $\ker\mathrm{d}\pi_1 = \ker(\mathrm{d}\phi_1 \circ \mathrm{d}\wt \pi_1)$. Note that in \hyperref[eq:sep3]{diagram~\ref*{eq:sep3}} the identification $\fg^{\times r} \cong \Hom(\cE, \fg)$ is given by sending an r-tuple $(g_1,\ldots, g_r)$ to the map $\cE \to \fg$ defined by $e_i \mapsto g_i$. Hence, the kernel of the vertical map is $\cE^{\times r} \subset \fg^{\times r}$. The map $\mathrm{d}\wt \pi_1$ is given by $g \mapsto ([g,e_1], \ldots, [g,e_r])$.  To land in $\ker\mathrm{d}\phi_1 = \cE^{\times r}$ we must have $[g,e_i] \in \cE$ for any $1\leq i \leq r$.  Hence, $\mathrm{d}\phi_1 \circ \mathrm{d}\wt \pi_1(g) = 0$ if and only if $g \in N_{\fg}(\cE)$.
\end{proof}

\begin{thm} \label{thm:sep2}
If $P = \Stab_G(\cE)$ is parabolic then the orbit $G\cdot\cE \subseteq \bE(\fg)$ is isomorphic to the flag variety $G/P$.
\end{thm}
\begin{proof}
By Borel~\cite[6.7]{borelLAG} and the above lemma this is equivalent to the statement that $N_{\fg}(\cE) \subseteq \lie(P)$.  Without loss of generality we assume $P$ is a standard parabolic, then $\cE$ is fixed by the torus and therefore is a Chevalley subalgebra.  The normalizer $N_{\fg}(\cE)$ is then a Chevalley subalgebra as well so it suffices to choose $\alpha \in \Phi$ and show that $x_\alpha \in N_{\fg}(\cE)$ implies $x_\alpha \in \lie(P)$.

We assume that $p$ is separably good for $G$ so, in particular, $p$ is greater or equal to the length of the longest root string in $\Phi$.  This implies that any structure constants which are zero in $k$ are zero in $\bZ$.  We get then that the $\bZ$-form of $x_\alpha$ normalizes the $\bZ$-form of $\cE$.  The action of the root space $U_\alpha$ on $\fg$ is given by exponentiating the adjoint action of the $\bZ$-form of $x_\alpha$ and then base changing to $k$, thus $U_\alpha$ stabilizes $\cE$.  As $x_\alpha$ spans the Lie algebra of $U_\alpha$ we have $x_\alpha \in \lie(P)$ as desired.
\end{proof}

\begin{example}
\label{ex:G2}
To see how the above argument can fail when $p$ is less than the maximal length of a root string in $\Phi$ consider $G$ of type $G_2$ and $p = 3$.  We saw in \autoref{sec:G2p3} that $\cE = \lie(R_1)$ is a maximal elementary abelian.  The stabilizer of $\cE$ is the Borel $B$ but one can check that in $\fg_{\bZ}$ we have $[x_{-\alpha_1}, \cE] = 3\bZ x_{x_{\alpha_2}} + \cE$, thus $x_{-\alpha_1}$ normalizes $\cE$ in $\fg_{\bF_3}$.  If $M_0 \in \mathbb M_{14}(\mathbb Z)$ is the matrix of $\ad(x_{-\alpha_1})$ in $\fg_{\mathbb Z}$ and $M_3 \in \mathbb M_{14}(\mathbb F_3)$ is its mod $3$ image, i.e., the matrix of $\ad(x_{-\alpha_1})$ in $\fg_{\bF_3}$, then $M_0^3 \neq 0$ but $M_3^3 = 0$.  Thus even though $\exp(M_0)$ and $\exp(M_3)$ are both well defined and $\exp(M_3)$ stabilizes $\cE$, the action of $U_{-\alpha_1}$ is given by the mod $3$ image of $\exp(M_0)$ and this does not equal $\exp(M_3)$.
\end{example}

\subsection{Majority case}

Retaining the notation from the previous section we now consider $G$ which is \emph{not} of type $A_2$ or $G_2$. 

\begin{thm}
\label{thm:main}  
Let $G$ be a simple algebraic group, not of type $A_2$ or $G_2$.  Assume that $p$ is separably good for $G$.  
Then 
\[\bE(\fg) = \coprod_{\substack{R \in \max(\Phi) \\ R \ \mathrm{an \ ideal}}}G/P_{R},\]
where $P_{R} = \Stab_G(\lie(R))$.
\end{thm}

\begin{proof}
From \autoref{thm:EgUnion} the variety $\bE(\fg)$ is a union of orbits $G\cdot\lie(R)$ where $R$ ranges over the ideals in $\max(\Phi)$ and \autoref{thm:sep2} gives that each orbit $G\cdot\lie(R)$ is isomorphic to the flag variety $G/P_{R}$.  These orbits are therefore closed and we need only prove that they are distinct.

By the Bruhat decomposition two such $\lie(R)$ are conjugate if and only if the corresponding $R$ are conjugate via the Weyl group.  At most one ideal in $\max(\Phi)$ is not of the form $\Phir{i}$ for some $i$ so \autoref{lem:radConj} gives that distinct maximal commuting ideals are non-conjugate as desired.
\end{proof}

Combining \autoref{thm:main} with \autoref{table:max} we get, except for types $A_2$ and $G_2$, the explicit type-by-type calculation of $\bE(\fg)$ found in \autoref{table:Eg}.

\begin{table}[ht]
\caption{$\bE(\fg)$ for $p$ separably good.}
\centering
\renewcommand{\arraystretch}{1.4}
\begin{tabular}[b]{|c|c|c|}
\hline
Type & \parbox{50pt}{\vspace{3pt}\centering Restrictions on rank\vspace{3pt}}  & $\bE(\fg)$ \\
\hline \hline
\multirow{2}{*}{$A_{2n}$} & $n = 1$ & \begin{tabular}{p{1.8in}} Irreducible, $5$-dimensional \end{tabular} \\
\cline{2-3}
& $n \geq 2$ & $G/P_{\Delta \setminus \set{\alpha_n}} \coprod G/P_{\Delta \setminus \set{\alpha_{n + 1}}}$ \\
\hline
$A_{2n+1}$ & $n \geq 0$ & $G/P_{\Delta \setminus \set{\alpha_{n + 1}}}$ \\
\hline
\multirow{3}{*}{$B_n$}& $n = 2,3$ & $G/P_{\Delta \setminus \set{\alpha_1}}$ \\
\cline{2-3}
& $n = 4$ & $G/P_{\Delta \setminus \set{\alpha_1}} \coprod G/P_{\set{\alpha_2, \alpha_3}}$ \\
\cline{2-3}
& $n \geq 5$ & $G/P_{\Delta \setminus \set{\alpha_1, \alpha_n}}$ \\
\hline
$C_n$ & $n \geq 3$ & $G/P_{\Delta \setminus \set{\alpha_n}}$ \\
\hline
\multirow{2}{*}{$D_n$} & $n = 4$ & $G/P_{\Delta \setminus \set{\alpha_1}} \coprod G/P_{\Delta \setminus \set{\alpha_3}} \coprod G/P_{\Delta \setminus \set{\alpha_4}}$ \\
\cline{2-3}
& $n \geq 5$ & $G/P_{\Delta \setminus \set{\alpha_{n - 1}}} \coprod G/P_{\Delta \setminus \set{\alpha_n}}$ \\
\hline
$E_6$ && $G/P_{\Delta \setminus \set{\alpha_1}} \coprod G/P_{\Delta \setminus \set{\alpha_6}}$ \\
\hline
$E_7$ && $G/P_{\Delta \setminus \set{\alpha_7}}$ \\
\hline
$E_8$ && $G/P_{\Delta \setminus \set{\alpha_2}}$ \\
\hline
$F_4$ && $G/P_{\set{\alpha_1, \alpha_3}}$ \\
\hline
$G_2$ && \begin{tabular}{p{1.8in}} Irreducible, $8$-dimensional \end{tabular} \\
\hline
\end{tabular}
\label{table:Eg}
\end{table}

\subsection{Type $A_2$}

In type $A_2$ there are $3$ disjoint orbits, two of which contain Chevalley subalgebras corresponding to ideals and one which does not contain a Chevalley subalgebra.  The following are representatives of the three orbits
\begin{align*}
L_1 &= \langle x_{\alpha_2}, x_{\alpha_1 + \alpha_2}\rangle, \\
L_2 &= \langle x_{\alpha_1}, x_{\alpha_1 + \alpha_2}\rangle, \\
L_3 &= \langle x_{\alpha_1} + x_{\alpha_2}, x_{\alpha_1 + \alpha_2}\rangle.
\end{align*}
The stabilizers of these subalgebras are
\begin{align*}
\Stab_G(L_1) &= P_1, \\
\Stab_G(L_2) &= P_2, \\
\Stab_G(L_3) &= \langle \alpha_1^\vee(\lambda_1)\alpha_2^\vee(\lambda_2), U \ | \ \lambda_1^3 = \lambda_2^3 \rangle,
\end{align*}
and so the orbits have dimensions
\begin{align*}
\dim(G\cdot L_1) &= 2, \\
\dim(G\cdot L_2) &= 2, \\
\dim(G\cdot L_3) &= 5.
\end{align*}
As $G\cdot L_1 \simeq G/P_1$ and $G\cdot L_2 \simeq G/P_2$ they are closed orbits and $G\cdot L_3$ is open.  The map $\mathbb P^1 \to \bE(\fg)$ given by $[a : b] \mapsto \langle ax_{\alpha_1} + bx_{\alpha_2}, x_{\alpha_1 + \alpha_2}\rangle$ then yields that the closure of $G\cdot L_3$ contains the other two orbits.  So $G\cdot L_3$ is dense, hence $\bE(\fg)$ is irreducible of dimension $5$.

\subsection{Type $G_2$}

In type $G_2$ there are again $3$ disjoint orbits, two of which contain Chevalley subalgebras and one which does not, but now only one orbit contains a Chevalley subalgebra corresponding to an ideal.  The following are representatives of the three orbits
\begin{align*}
L &= \langle x_{\alpha_2} + x_{3\alpha_1 + \alpha_2}, x_{2\alpha_1 + \alpha_2}, x_{3\alpha_1 + 2\alpha_2}\rangle \\
\lie(C_3) &= \langle x_{\alpha_2}, x_{2\alpha_1 + \alpha_2}, x_{3\alpha_1 + 2\alpha_2}\rangle, \\
\lie(C_5) &= \langle x_{2\alpha_1 + \alpha_2}, x_{3\alpha_1 + \alpha_2}, x_{3\alpha_1 + 2\alpha_2}\rangle,
\end{align*}
and $C_5$ is the ideal.  The stabilizers of these subalgebras are
\begin{align*}
\Stab_G(L) &= \langle\alpha_1^\vee(\lambda_1)\alpha_2^\vee(\lambda_2), U_\alpha \ | \ \lambda_1^2 = \lambda_2^3, \ \alpha \in \Phi^+ \setminus \set{\alpha_1}\rangle, \\
\Stab_G(\lie(C_3)) &= \langle T, U_\alpha \ | \ \alpha \in \Phi^+ \setminus \set{\alpha_1}\rangle, \\
\Stab_G(\lie(C_5)) &= P_2,
\end{align*}
so the orbits have dimensions
\begin{align*}
\dim(G\cdot L) &= 8, \\
\dim(G\cdot\lie(C_3)) &= 7, \\
\dim(G\cdot\lie(C_5)) &= 5.
\end{align*}
As $\bE(\fg)$ is a closed subvariety of the Grassmannian it is complete.  The Borel fixed point theorem then gives that $G\cdot\lie(C_5)$ is the only closed orbit.  Boundaries of orbits are unions of orbits of smaller dimension therefore the closure of $G\cdot\lie(C_3)$ is its union with $G\cdot\lie(C_5)$ and the orbit $G\cdot L$ is open.  To see that $G\cdot\lie(C_3)$ is not open we need that the closure of $G\cdot L$ contains $G\cdot\lie(C_3)$; one sees this from the map $\mathbb A^1 \to \bE(\fg)$ defined by $a \mapsto \langle x_{\alpha_2} + ax_{3\alpha_1 + \alpha_2}, x_{2\alpha_1 + \alpha_2}, x_{3\alpha_1 + 2\alpha_2}\rangle$ which sends $\mathbb A^1 \setminus \set{0}$ into $G\cdot L$ and $0$ to $\lie(C_3)$.  In particular, we have now shown that $G\cdot L$ is dense and so $\bE(\fg)$ is irreducible of dimension $8$.


\section{Applications to Chevalley groups}
\label{sec:groups}

Let $k = \overline{\bF}_p$ be the algebraic closure of $\bF_p$ and let $G$ be a reductive $k$-group, defined and split over $\bZ$ as above.  
Our calculation of $\bE(\fg)$ yields information about the maximal elementary abelian $p$-subgroups of the Chevalley groups $G(\bF_{q})$ for $q=p^r$ and their conjugacy classes.
Since any such subgroup can be conjugated into the Sylow $p$-subgroup $U(\bF_q)$, we can assume that $G$ is semi-simple simply connected. Assume $p$ is good for $G$, and let $\cU_1(G)$ and $\cN_1(\fg)$ be the varieties of $p$-unipotent and $p$-nilpotent elements in $G$ and $\fg$ respectively.

To translate between the Lie algebra and Chevalley group we use the theorem below.  

\begin{thm}[{\cite[4.3]{sobajeSpringerIsos}}] \label{thm:Springer}
There exists a unique isomorphism $\phi\colon\cN_1(\fg) \to \cU_1(G)$ satisfying
\begin{enumerate}
\item $\phi$ is $G$-equivariant,
\item $x, y \in \fg$ commute if and only if $\phi(x), \phi(y) \in G$ commute,
\item $x \in \fg_{\bF_q}$ if and only if $\phi(x) \in G(\bF_q)$.
\end{enumerate}
\end{thm}

Such a $\phi$ gives an inclusion preserving bijection between $p$-nilpotent commutative subsets of $\fg_{\bF_q}$ and $p$-unipotent commutative subsets of $G(\bF_q)$.  A maximal commuting set of $p$-unipotent elements in $G(\bF_q)$ is necessarily a maximal elementary abelian subgroup.  Similarly a maximal set of commuting $p$-nilpotent elements in $\fg_{\bF_q}$ is necessarily a maximal elementary subalgebra, and therefore corresponds to a maximal elementary subalgebra of $\fg$ that is defined over $\bF_q$.  As the $\bF_q$-rational points of the Grassmannian are exactly the subspaces defined over $\bF_q$ we now have the following.

\begin{thm}
\label{thm:conjFq}
The map $\phi$ induces a $G(\bF_q)$-equivariant bijection between the $\bF_q$-rational points of $\bE(\fg)$ and the maximal elementary abelian subgroups of $G(\bF_q)$.
\end{thm}

Even though we don't need this observation for our application to Chevalley groups, the following enhancement of \autoref{cor:conj} is worthy of pointing out:

\begin{cor}
Let $\cE$ be an $\bF_q$-rational point of $\bE(\fg)$.  Then $\cE$ is $G(\bF_q)$-conjugate to a subalgebra of $\fu$.
\end{cor}
\begin{proof}
The subalgebra $\cE$ corresponds, using $\phi$, to a maximal elementary abelian subgroup of $G(\bF_q)$.  As $U(\bF_q)$ is a $p$-Sylow subgroup of $G(\bF_q)$ an element $g \in G(\bF_q)$ conjugates this elementary abelian into $U(\bF_q)$.  The equivariance of $\phi$ then gives that $g$ conjugates $\cE$ into $\fu$.
\end{proof}

\hyperref[thm:EuIdeal]{Theorem~\ref*{thm:EuIdeal}} together with the calculation for $A_2$ in  \autoref{sec:unip} show that in all types except for $G_2$ if two elementary subalgebras in $\bE(\fu)$ are defined over $\bF_q$ and conjugate by an element in $G(k)$ then they are already conjugate by an element in $G(\bF_q)$.  
This observation, combined with \autoref{thm:conjFq}, allows us to translate the results for $\bE(\fu)$ from \autoref{sec:unip} to the classification of conjugacy classes of maximal elementary abelian $p$-subgroups of $G(\bF_q)$. Consequently, we recover the results of Barry \cite{barryLargeAbelianSubgroups} on maximal elementary abelian $p$-subgroups of Chevalley groups of classical type, and supplement Barry's results with similar information for the exceptional types.

\begin{thm}
Let $G$ be a simple algebraic group defined and split over $\bZ$. Assume $p$ is good for $G$.
Then the conjugacy classes and ranks of the elementary abelian $p$-subgroups of $G(\mathbb F_q)$ of maximal rank are given in \autoref{table:Chevalley}, where in types $E_8$ and $F_4$ we take $R$ to be the unique ideal in $\max(\Phi)$.
\end{thm}
Note that for $G$ reductive, the representatives of conjugacy classes of maximal elementary abelian $p$-subgroups of $G(\bF_q)$ are given by products of representatives for each simple entry in the direct product decomposition of the derived group $[G,G]$. 
\renewcommand{\arraystretch}{1.5}

\begin{table}[ht]
\caption{Maximal elementary abelian subgroups of $G(\mathbb F_q)$.}
\centering
\begin{tabular}[b]{|c|c|c|c|c|}
\hline
Type & \parbox{50pt}{\vspace{3pt}\centering Restrictions on rank\vspace{3pt}} & \parbox{55pt}{\vspace{3pt}\centering \# conjugacy classes\vspace{3pt}} & order & Representatives \\
\hline \hline
\multirow{3}{*}{$A_{2n}$} & $n = 1$ & $3$ & $q^2$ & $\begin{array}{c} \langle U_{\alpha_1}, U_{\alpha_1+\alpha_2} \rangle \\ \langle U_{\alpha_2}, U_{\alpha_1+\alpha_2} \rangle \\ \langle \phi(\bF_q(x_{\alpha_1} + x_{\alpha_2})), U_{\alpha_1 + \alpha_2} \rangle  \end{array}$ \\
\cline{2-5}
& $n \geq 2$ &  $2$ & $q^{n(n+1)}$ & $\begin{array}{c} \langle U_\alpha(\bF_q) \ | \ \alpha \in \Phir{n+1}\rangle \\ \langle U_\alpha(\bF_q) \ | \ \alpha \in \Phir{n}\rangle \end{array}$ \\
\hline
$A_{2n+1}$ & $n \geq 0$ &  $1$ & $q^{(n+1)^2}$ & $\langle U_\alpha(\bF_q) \ | \ \alpha \in \Phir{n+1}\rangle$\\
\hline
\multirow{4}{*}{$B_n$}& $n = 2,3$ &  $1$ & $q^{2n-1}$ & $\langle U_\alpha(\bF_q) \ | \ \alpha \in \Phir{1}\rangle$ \\
\cline{2-5}
& $n = 4$ & $2$ & $q^7$& 
$\begin{array}{c} \langle U_\alpha(\bF_q) \ | \ \alpha \in \Phir{1}\rangle \\ \langle U_\alpha(\bF_q) \ | \ \alpha \in S_1\rangle \end{array}$\\
\cline{2-5}
& $n \geq 5$ & $1$ & $q^{\frac{1}{2}n(n-1)+1}$& $\langle U_\alpha(\bF_q) \ | \ \alpha \in S_1\rangle$\\
\hline
$C_n$ & $n \geq 2$  & $1$ & $q^{\frac{1}{2}n(n+1)}$ & $\langle U_\alpha(\bF_q) \ | \ \alpha \in \Phir{n}\rangle$\\
\hline
\multirow{3}{*}{$D_n$}& $n=4$ & $3$ & $q^6$ &
$\begin{array}{c} \langle U_\alpha(\bF_q) \ | \ \alpha \in \Phir{1}\rangle \\ \langle U_\alpha(\bF_q) \ | \ \alpha \in \Phir{3}\rangle \\ \langle U_\alpha(\bF_q) \ | \ \alpha \in \Phir{4}\rangle \end{array}$\\
\cline{2-5}
& $n \geq 5$ & $2$ & $q^{\frac{1}{2}n(n-1)}$& 
$\begin{array}{c} \langle U_\alpha(\bF_q) \ | \ \alpha \in \Phir{n-1}\rangle \\ \langle U_\alpha(\bF_q) \ | \ \alpha \in \Phir{n}\rangle \end{array}$ \\
\hline
$E_6$ && $2$ & $q^{16}$ & $\begin{array}{l} \langle U_\alpha(\bF_q) \ | \ \alpha \in \Phir{1}\rangle \\ \langle U_\alpha(\bF_q) \ | \ \alpha \in \Phir{6}\rangle \end{array}$ \\
\hline
$E_7$ && $1$ & $q^{27}$ & $\langle U_\alpha(\bF_q) \ | \ \alpha \in \Phir{7}\rangle$ \\
\hline
$E_8$ && $1$ & $q^{36}$ & $\langle U_\alpha(\bF_q) \ | \ \alpha \in I\rangle$ \\
\hline
$F_4$ && $1$ & $q^9$ & $\langle U_\alpha(\bF_q) \ | \ \alpha \in I\rangle$ \\
\hline
$G_2$ && $\geq 3$ & $q^3$ & \\
\hline
\end{tabular}
\label{table:Chevalley}
\end{table}

\begin{example} In type $G_2$ we used the fact that $k$ is algebraically closed in a nontrivial way in \autoref{sec:G2unip} to conclude that $\bE(\fg)$ had three $G$-orbits.  Thus we can only conclude that there are \emph{at least} three conjugacy classes of maximal elementary abelian subgroups in $G(\bF_q)$.  In fact the number of conjugacy classes depends on $q$.  For example, if $q = 5^r$ then $G(\bF_q)$ has three conjugacy classes of elementary abelian subgroups when $r$ is odd but has six such conjugacy classes when $r$ is even.
\end{example}

As implied by the Quillen stratification theorem~\cite{quillen} the above calculation gives the number of irreducible components of maximal dimension for $\Spec H^*(G(\bF_q), k)$.

\begin{cor}
Let $G$ be a simple algebraic group defined and split over $\bZ$. Assume $p$ is good for $G$. Then the dimension of $\Spec H^*(G(\bF_q), k)$ and the number of irreducible components of maximal dimension are given in \autoref{table:spec}. 
\begin{table}[ht]
\caption{ $\Spec H^*(G(\mathbb F_{p^r}), k)$.}
\centering
\begin{tabular}[b]{|c|c|c|c|}
\hline
\rm Type & \parbox{50pt}{\vspace{3pt}\centering\rm Restrictions on rank\vspace{3pt}}  & \parbox{70pt}{\vspace{3pt}\centering\rm \# of irreducible components of max dimension\vspace{3pt}} & \rm dimension \\
\hline \hline
\multirow{2}{*}{$A_{2n}$} & $n = 1$ & $3$ & $p^{2r-1}$ \\
\cline{2-4}
& $n \geq 2$ &  $2$ & $p^{rn(n+1)-1}$ \\
\hline
$A_{2n+1}$ & $n \geq 0$ & $1$ & $p^{r(n+1)^2-1}$ \\
\hline
\multirow{3}{*}{$B_n$}& $n=2,3$ &  $1$ & $p^{r(2n-1)-1}$ \\
\cline{2-4}
& $n = 4$ & $1$ & $p^{7r-1}$ \\
\cline{2-4}
& $n \geq 4$ & $1$ & $p^{\frac{1}{2}rn(n-1)+r-1}$ \\
\hline
$C_n$ & $n \geq 2$ & $1$ & $p^{\frac{1}{2}n(n+1)r-1}$ \\
\hline
\multirow{2}{*}{$D_n$}& $n=4$ & $3$ & $p^{6r-1}$ \\
\cline{2-4}
& $n \geq 5$ & $2$ & $p^{\frac{1}{2}n(n-1)r-1}$ \\
\hline
$E_6$ && $2$ & $p^{16r-1}$ \\
\hline
$E_7$ && $1$ & $p^{27r-1}$ \\
\hline
$E_8$ && $1$ & $p^{36r-1}$ \\
\hline
$F_4$ && $1$ & $p^{9r-1}$ \\
\hline
$G_2$ && $\geq 3$ & $p^{3r-1}$ \\
\hline
\end{tabular}
\label{table:spec}
\end{table}
\end{cor}
\begin{proof} 
By Quillen~\cite{quillen} the dimension of $\Spec H^*(G(\bF_q), k)$ equals the maximal rank of an elementary abelian $p$-subgroup of $G(\bF_q)$. The maximal rank is the maximal order divided by $p$ and, hence, can be read from \autoref{table:Chevalley}.  The Quillen stratification theorem implies that the number of irreducible components equals the number of conjugacy classes of maximal elementary abelian $p$-subgroups which once again can be read from \autoref{table:Chevalley}. 
\end{proof}


\appendix
\addtocontents{toc}{\protect\setcounter{tocdepth}{1}}
\section{Maximal sets of commuting roots} \label{Appendix}

In the Appendix we give some details for the description of the maximal subsets of commuting roots found in \autoref{table:max} and the stabilizers of certain ideals in \autoref{table:stab}.  Note that the maximal subsets of commuting roots are computed in Malcev~\cite{malcev} except that he skips the proof for $E_8$ and the paper is in Russian.

We begin with the computation of maximal subsets of commuting roots.  For type $E$, $D_n$ when $n < 7$, and $B_n$ when $n < 5$ we use a computer program which we have made available online~\cite{starkCommRoots}.  For the remaining types we provide the following arguments.

\subsection{Type $A_n$}

One can check, as in Grantcharov and Serganova~\cite{grantSerg}, that sending $J \subseteq \set{1, 2, \ldots, n + 1}$ to the set of roots $\set{\epsilon_i - \epsilon_j \ | \ i \in J, j \notin J}$ yields a bijection between proper nontrivial subsets of $\set{1, 2, \ldots, n + 1}$ and sets of inclusion maximal commutative subsets of $\Phi$.  As $J$ gets sent to a set of size $|J|(n + 1 - |J|)$ we see that this set is of maximal order when $n = 2m$ and $|J| = m, m + 1$ or when $n = 2m + 1$ and $|J| = m + 1$.  It is a set of positive roots if and only if $J < \set{1, \ldots, n + 1} \setminus J$, thus in type $A_{2m}$ we have $J = \set{1, \ldots, m}$ or $\set{1, \ldots, m + 1}$ yielding $\Phir m$ and $\Phir{m + 1}$, respectively, and in type $A_{2m + 1}$ we have $\set{1, \ldots, m + 1}$ yielding $\Phir{m + 1}$.

\subsection{Type $B_n$}

We assume $n \geq 5$.  The set $R = \set{\epsilon_i \ | \ 1 \leq i \leq n}$ is an inclusion maximal set of non-commuting roots so any maximal set of commuting roots in $\Phi^+$ consists of a maximal set of commuting roots in $\Phi^+ \setminus R$ together with at most one element from $R$.  Observe that $\Psi = \Phi \setminus \pm R$ is a root system of type $D_n$ with simple roots $\set{\alpha_1, \ldots, \alpha_{n - 1}, \alpha_{n - 1} + 2\alpha_n}$.  The maximal set $\Psi^\mathrm{rad}_n$ can commute with any $\epsilon_i$ and yields $S_i$.  The maximal set $\Psi^\mathrm{rad}_{n - 1}$ commutes with $\epsilon_i$ when $i < n$ and yields $S_i^\ast$.

\subsection{Type $C_n$}

Sending $J \subseteq \set{1, \ldots, n}$ to the set
\[\phi(J) = \set{\epsilon_i + \epsilon_{i'}, \epsilon_i - \epsilon_j, -\epsilon_j - \epsilon_{j'} \ | \ i, i' \in J \ \text{and} \ j, j' \notin J}\]
gives a bijection $\phi$ between the power set of $\set{1, \ldots, n}$ and inclusion maximal unipotent commuting subsets of $\Phi$.  Among those subsets $J$ satisfying $|J| = m$, the number of positive roots in $\phi(J)$ attains a maximum of $\frac12m(m + 1) + m(n - m)$ when $J < \set{1, \ldots, n} \setminus J$ and this maximum value for a given $m$ attains a maximum of $\frac12n(n + 1)$ when $m = n$.  Thus we take the positive roots of $\phi(\set{1, \ldots, n})$ and get $\Phir n$.

\subsection{Type $D_n$}

We assume $n \geq 7$.
\begin{lemma}
Let $\Phi$ be type $D_n$.  If $R \subseteq \Phir{\alpha_1, \alpha_2}$ is an inclusion maximal set of commuting roots which contains $\epsilon_1 - \epsilon_2$ then $R = \Phir1$ has order $2n - 2$.  If it does not contain $\epsilon_1 - \epsilon_2$ then it consists of the root $\epsilon_1 + \epsilon_2$ together with one choice of root from each of the sets $\set{\epsilon_1 + \epsilon_r, \epsilon_2 - \epsilon_r}_{2 < r \leq n}$ and $\set{\epsilon_1 - \epsilon_r, \epsilon_2 + \epsilon_r}_{2 < r \leq n}$, and hence has order $2n - 3$.
\end{lemma}
\begin{proof}
We have $\Phir{\alpha_1, \alpha_2} = \set{\epsilon_1 \pm \epsilon_i, \epsilon_2 \pm \epsilon_j \ | \ 2 \leq i \leq n \ \text{and} \ 3 \leq j \leq n}$.  If $\epsilon_1 - \epsilon_2$ is contained in our maximal set then the roots $\epsilon_2 \pm \epsilon_j$ are not, so the set contains at most the roots $\epsilon_1 \pm \epsilon_i$, i.e., the roots of $\Phir1$.  These indeed commute and there are $2n -2$ of them.  If $\epsilon_1 - \epsilon_2$ is not contained in our maximal set then note that $\epsilon_1 + \epsilon_2$ is the longest root and therefore is contained in any inclusion maximal set of commuting roots.  The remaining roots form the sets of non-commuting pairs given in the statement.  One sees that roots from distinct pairs commute and there are $2n - 4$ such pairs.
\end{proof}

Observe that $m(\Phi) \geq |\Phir n| = \frac12n(n - 1)$.  Also $\Phir1$ is inclusion maximal and of smaller order so no element of $\max(\Phi)$ contains $\Phir1$.  Now $\Psi = \Phi \setminus \pm\Phir{\alpha_1, \alpha_2}$ is a root system of type $D_{n - 2}$ with simple roots $\set{\alpha_3, \ldots, \alpha_n}$.  Every set of commuting roots in $\Phi^+$ is the union of sets of commuting roots from $\Psi^+$ and $\Phir{\alpha_1, \alpha_2}$.  By the lemma above the set of commuting roots from $\Phir{\alpha_1, \alpha_2}$ can have at most $2n - 3$ elements and by induction the set from $\Psi^+$ can have at most $\frac12(n - 2)(n - 3)$.  These sum to the order of $\Phir n$ so a maximal set of commuting roots must be the union of a maximal set from $\Psi^+$ and a set of order $2n - 3$ from $\Phir{\alpha_1, \alpha_2}$.

Now it suffices to take $R \in \max(\Psi)$ and check which roots in $\Phir{\alpha_1, \alpha_2}$ it commutes with.  If $R = \Psi^\mathrm{rad}_n$ then $\epsilon_i + \epsilon_n \in R$ for all $2 < i < n$ so $\epsilon_1 - \epsilon_j, \epsilon_2 - \epsilon_j \notin R$ for all $2 < j \leq n$.  This identifies a unique inclusion maximal set of commuting roots in $\Phir{\alpha_1, \alpha_2}$ and it's union with $\Psi^\mathrm{rad}_n$ is $\Phir n$.  If $R = \Psi^\mathrm{rad}_{n - 1}$ then $\epsilon_i - \epsilon_n \in R$ for all $2 < i < n$ so $\epsilon_1 - \epsilon_i, \epsilon_2 - \epsilon_i, \epsilon_1 + \epsilon_n, \epsilon_2 + \epsilon_n \notin R$.  Again this identifies the inclusion maximal set in $\Phir{\alpha_1, \alpha_2}$ and it's union with $\Psi^\mathrm{rad}_{n - 1}$ is $\Phir{n - 1}$.

\subsection{Type $F_4$}

Recall that the roots of $F_4$ are
\[\pm\epsilon_i, \pm\epsilon_i \pm \epsilon_j, \frac12(\pm\epsilon_1 \pm \epsilon_2 \pm \epsilon_3 \pm \epsilon_4).\]
We will denote positive roots of the last type by $\epsilon_{ijk} = \frac12(\epsilon_1 + i\epsilon_2 + j\epsilon_3 + k\epsilon_4)$ where $i, j, k \in \set{\pm1}$ and will write, for example, $\epsilon_{+-+}$ instead of $\epsilon_{1, -1, 1}$.
\begin{lemma}
If $\epsilon_{ijk} \neq \epsilon_{i'j'k'}$ then $\epsilon_{ijk}$ and $\epsilon_{i'j'k'}$ commute if and only if there is exactly one sign change between $(i, j, k)$ and $(i', j', k')$.  In particular, a commuting set of roots can have at most $2$ roots of the form $\epsilon_{ijk}$.
\end{lemma}
\begin{proof}
Observe that the terms in $\epsilon_{ijk} + \epsilon_{i'j'k'}$ are exactly those $e_t$ for which the sign did not change (including $\epsilon_1$).  If there are $1$ or $2$ such terms then $\epsilon_{ijk} + \epsilon_{i'j'k'}$ is a root.  As the roots are distinct there cannot be $4$ such terms, therefore for $\epsilon_{ijk}$ and $\epsilon_{i'j'k'}$ to commute there must be $3$ such terms, hence exactly one sign change.

If $(i, j, k)$, $(i', j', k')$, and $(i'', j'', k'')$ are mutually distinct and there is exactly one sign change from $(i, j, k)$ to $(i', j', k')$ and $(i'', j'', k'')$ then there are $2$ sign changes from $(i', j', k')$ to $(i'', j'', k'')$.  This proves that $3$ roots of the form $(i, j, k)$ cannot pairwise commute.
\end{proof}
Now observe that the roots $\pm\epsilon_i$ and $\pm\epsilon_i \pm \epsilon_j$ give $B_4 \subseteq F_4$.  As the commuting property is preserved when intersecting with a subroot system the lemma above gives that a maximal set of commuting roots in $F_4$ can have at most $9$ roots: $7$ from a maximal set in $B_4$ plus $2$ additional roots of the form $\epsilon_{ijk}$.  This maximum is indeed attained so every maximal set of commuting positive roots in $F_4$ is identified by a tripel $(C, ijk, i'j'k')$, where $C \subseteq B_4$ is a maximal set of commuting roots and $\epsilon_{ijk}$ and $\epsilon{i'j'k'}$ are the two additional roots.  We now compute all the possibilities.

\subsubsection{$C = \Phir{1}$}

Every $\epsilon_{ijk}$ commutes with $C = \set{\epsilon_1, \epsilon_1 \pm \epsilon_i \ | \ i = 2, 3, 4}$.  Once $(i, j, k)$ are chosen there are three tripels $(i', j', k')$ which differ by a single sign change.  This gives $24$ ordered pairs of roots $(\epsilon_{ijk}, \epsilon_{i'j'k'})$ that can be added, thus $12$ possible sets of roots for this case.

\subsubsection{$C = S_t$}

We have $C = \set{\epsilon_t, \epsilon_i + \epsilon_j \ | \ 1 \leq i < j \leq 4}$ and for $\epsilon_{ijk}$ to commute with $\epsilon_t$ the sign on $\epsilon_t$ must be positive.  We cannot have more than one negative sign in $(i, j, k)$ otherwise $\epsilon_{ijk}$ would not commute with some root of the form $\epsilon_i + \epsilon_j$.  Thus the two roots of the form $\epsilon_{ijk}$ must be $\epsilon_{+++}$ and $\epsilon_{ijk}$ where there is exactly one negative sign in $(i, j, k)$ and this negative sign is not on the $\epsilon_t$ term.

Thus for $C = S_1$ we get three maximal sets corresponding to the three choices for a negative sign and for $C = S_2, S_3, S_4$ we get two maximal sets each.  This gives $9$ possible sets of roots for this case.

\subsubsection{$C = S^\ast_t$}

We have $C = \set{\epsilon_t, \epsilon_i + \epsilon_j, \epsilon_{i'} - \epsilon_4 \ | \ 1 \leq i < j < 4, 1 \leq i' < 4}$.  Because of the $\epsilon_{i'} - \epsilon_4$ terms the only $\epsilon_{ijk}$ with a single negative that commutes with $C$ is $\epsilon_{++-}$.  As $\epsilon_{--+}$ and $\epsilon_{---}$ don't commute with $\epsilon_2 + \epsilon_3$ we find that the two additional elements must be $\epsilon_{++-}$ and $\epsilon_{+++}$ or $\epsilon_{++-}$ and $\epsilon{ij-}$ where exactly one of $i, j$ is negative and the negative is not on the $\epsilon_t$ term.

The first choice is valid for any $t$.  For the second when $C = S_1$ there are two choices for the additional negative and when $C = S_2, S_3$ there is one choice for the additional negative.  This gives $7$ possible sets of roots for this case.

\subsection{Type $G_2$}

There are $3$ short and $3$ long positive roots.  No pair of short positive roots commute so there can be at most $1$ short root in a maximal commuting set.  The pair of long roots $(\alpha_2, 3\alpha_1 + \alpha_2)$ does not commute so a maximal set contains the highest root $3\alpha_1 + 2\alpha_2$ together with at most $1$ other long and $1$ short root.  With this one can check that the maximal sets are
\begin{align*}
& \set{\alpha_1, 3\alpha_1 + \alpha_2, 3\alpha_1 + 2\alpha_2}, \\
& \set{\alpha_1 + \alpha_2, 3\alpha_1 + \alpha_2, 3\alpha_1 + 2\alpha_2}, \\
& \set{\alpha_2, 2\alpha_1 + \alpha_2, 3\alpha_1 + 2\alpha_2}, \\
& \set{\alpha_2, \alpha_1 + \alpha_2, 3\alpha_1 + 2\alpha_2}, \\
& \set{2\alpha_1 + \alpha_2, 3\alpha_1 + \alpha_2, 3\alpha_1 + 2\alpha_2}. \\
\end{align*}

\bibliographystyle{alphanum}
\bibliography{refs}

\end{document}